\documentclass{article}

\usepackage{amsmath}
\usepackage{amsthm}
\usepackage[utf8]{inputenc}
\usepackage{amsfonts}
\usepackage{amssymb}
\usepackage{enumitem}
\usepackage{color}
\usepackage{hyperref}
\usepackage{MnSymbol}
\usepackage{graphicx}
\usepackage{stmaryrd}

\newcommand{\phieps}{\phi_{\varepsilon}}
\newcommand{\vareps}{\varepsilon}
\newcommand{\oeps}{\Omega_{\varepsilon}}

\newcommand{\btxfxe}{\left(t,\bar{x},\frac{x}{\varepsilon}\right)}
\newcommand{\x}{\bar{x}}

\newcommand{\rats}{\overset{2}{\rightharpoonup}}

\newcommand{\fxe}{\frac{x}{\varepsilon}}
\newcommand{\ie}{i.\,e., }

\newcommand{\R}{\mathbb{R}}
\newcommand{\Z}{\mathbb{Z}}
\newcommand{\N}{\mathbb{N}}
\newcommand{\geps}{\Gamma_{\varepsilon}}
\newcommand{\veps}{v_{\varepsilon}}

\newcommand{\ueps}{u_{\varepsilon}}
\newcommand{\peps}{p_{\varepsilon}}
\newcommand{\tueps}{\tilde{u}_{\varepsilon}}
\newcommand{\tveps}{\tilde{v}_{\varepsilon}}
\newcommand{\oem}{\Omega_{\varepsilon}^M}
\newcommand{\weps}{w_{\varepsilon}}

\newcommand{\oef}{\Omega_{\varepsilon}^f}

\newcommand{\sepm}{S_{\varepsilon}^{\pm}}

\newcommand{\oems}{\Omega_{\varepsilon}^{M,s}}
\newcommand{\oemf}{\Omega_{\varepsilon}^{M,f}}

\newcommand{\foe}{\dfrac{1}{\varepsilon}}
\newcommand{\bxfxe}{\left(\bar{x},\dfrac{x}{\varepsilon}\right)}
\newcommand{\spaceH}{\mathcal{H}}
\newcommand{\spaceV}{\mathcal{V}}
\newcommand{\Div}{\mathrm{Div}}

\newtheorem{definition}{Definition}
\newtheorem{remark}{Remark}
\newtheorem{theorem}{Theorem}
\newtheorem{proposition}{Proposition}
\newtheorem{lemma}{Lemma}
\newtheorem{corollary}{Corollary}

\title{Effective interface laws of Navier-slip-type involving the elastic displacement for Stokes flow through a thin porous elastic layer}
\author{Markus Gahn and Maria Neuss-Radu}
\date{July 2025}

\begin{document}

\maketitle

\begin{center}
    \textit{Dedicated to Willi J\"ager on the anniversary of his 85th birthday.}
\end{center}
\smallskip

\begin{abstract}
This paper presents a rigorous derivation of an effective model for fluid flow through a thin elastic porous membrane separating two fluid bulk domains. The microscopic setting involves a periodically structured porous membrane composed of a solid phase and fluid-filled pores, with  thickness and periodicity of order $\vareps$, small compared to the size of the bulk regions. The microscopic model is governed by a coupled fluid-structure interaction system: instationary Stokes equations for the fluid and linear elasticity for the solid, with two distinct scalings of the elastic stress tensor yielding different effective behaviors.

Using two-scale convergence techniques adapted to thin domains and oscillatory microstructures, the membrane is reduced to an effective interface across which transmission conditions are derived. The resulting macroscopic model couples the bulk fluid domains via effective interface laws of Navier-slip-type including the dynamic displacement. The character of this coupling depends critically on the choice of the scaling in the elastic stress tensor, leading to either a membrane equation or a Kirchhoff-Love plate equation for the effective displacement.
The resulting interface conditions naturally admit mass exchange between the adjacent fluid regions. In the analytical framework, a new two-scale compactness theorem for the symmetric gradient is established, underpinning the passage to the limit in the coupled system. Moreover, cell problem techniques are employed systematically to construct admissible test functions and to rigorously extract the effective macroscopic coefficients.
\end{abstract}

\section{Introduction}

Fluid flow through elastic porous membranes separating bulk domains occurs in many applications from biology, medicine, material sciences and engineering, where it is often also coupled to mass or heat transport \cite{Huh2010,kuan2024fluid,SilvaJaeger2020}. Mathematical descriptions for such processes often treat the membrane as an interface where effective interface laws are used to describe the transition between the bulk domains. The rigorous derivation of such effective interface laws by methods of multi-scale analysis and dimension reduction is of crucial importance for the applications, as it yields effective parameters which can be calculated from standard problems. This allows the microscopic geometry and the processes inside the layer to be captured in the effective laws.

\textit{Willi J\"ager}, to whom we dedicate this paper, is one of the pioneers in the analysis of multiscale systems, especially in the rigorous derivation of transition laws at interfaces coupling different regimes. We mention here the seminal papers \cite{JaegerMikelic1996, mikelic2000interface}, where, together with Andro Mikeli\'c, he rigorously derived the law of Beavers-Joseph \cite{BeaversJoseph1967}, coupling free fluid flow with flow in porous bulk domains, by constructing accurate multiscale approximations based on boundary layers. The developed techniques were also successfully applied for the derivation of effective boundary conditions for viscous flows over rough boundaries \cite{jager2001roughness}. Properties of the effective models were analyzed numerically in \cite{jager2001asymptotic}. The mentioned contributions considered the case of rigid porous media. However, \textit{Willi J\"ager} and his collaborators also made important contributions to multiscale analysis of poroelastic systems. In particular, they promoted the application of poroelastic systems coupled with reaction-diffusion-transport problems. E.g., in \cite{JaegerMikelicNeussRaduAnalysis,JaegerMikelicNeussRaduHomogenization} fluid-structure interaction in cell tissues is coupled with diffusion, transport and reactions in the cells and the extra-cellular space. Passing to a scale limit, a quasi-static Biot system coupled with the upscaled reactive flow is obtained. Effective Biot’s coefficients depend on the reactant concentration. More recently, in \cite{gahn2022derivation}, \textit{Willi J\"ager} and his collaborators presented the first rigorous derivation of effective interface laws at the interface between two fluid bulk domains separated by an elastic porous membrane. For a specific scaling, in the limit a Kirchhoff-Love equation for the effective elastic displacement, together with a no slip condition for the effective fluid velocity was obtained. 

In this paper we consider, similar to \cite{gahn2022derivation}, fluid flow through an elastic porous membrane separating two fluid bulk domains. We consider a different scaling of the microscopic model with the aim of obtaining an effective model that allows mass transport across the membrane. This property is important for applications and was not provided by the effective model in \cite{gahn2022derivation}, even after adding first-order correctors to the effective approximation. In \cite{gahn2025effective} fluid flow through a rigid porous membrane was treated for a scaling of the flow equations which enabled mass flow through the membrane for the effective model. The effective interface laws for the velocity consisted in the continuity of the normal component and a slip  condition of Navier-type (similar to the Beavers-Joseph interface condition obtained for coupling between free fluid and porous bulk domains) for the tangential component on both sides of the interface. The aim of this paper is to rigorously derive effective interface laws for a fluid-structure interaction problem in the thin layer, and, in particular, to investigate how the interfacial laws in \cite{gahn2025effective} are affected when the rigid membrane is replaced by an elastic one. 

The microscopic fluid-structure interaction model considered in this paper is formulated in a domain consisting of two bulk regions separated by a porous membrane made up of a solid part and a pore space filled with fluid. The thickness of the membrane is of order $\vareps$ and the pore structure is periodic in tangential direction also with period $\vareps$. The microscopic interface between the two phases of the membrane, denoted by  $\Gamma_\vareps$ also has an $\vareps$-periodic structure.
The fluid flow in the bulk regions and in the pores of the membrane is described by an instationary Stokes system. The displacement of the solid part of the membrane is given by linear elasticity equations, in particular we consider two different scalings of the elastic stress tensor by $\varepsilon^{-\gamma}$, for $\gamma = 1,3$, leading to different effective behavior of the displacement. These two subsystems are coupled via transmission conditions at the microscopic fluid-solid interface $\Gamma_\vareps$, which consist in the continuity of the velocity and continuity of the normal stress. At the lateral boundaries of the domain we impose homogeneous Dirichlet boundary conditions for the velocity and for the displacement (in contrast to \cite{gahn2025effective}, where the microscopic velocity was assumed to satisfy periodic boundary conditions). This microscopic model is homogenized in the limit $\vareps \to 0$ by rigorous methods of two scale analysis and dimension reduction. In the limit we obtain an \textit{effective} or \textit{macroscopic model}, where the porous membrane is reduced to an effective interface $\Sigma$ (separating the two bulk regions) for which effective interface laws are derived. It has the advantage of reduced (computational) complexity, but still maintains key features about the microgeometry and processes in the membrane. 

The effective model consists of the instationary Stokes equations in the bulk regions coupled by effective transmission conditions at the interface $\Sigma$, and an effective system for the displacement which highly depends on the choice of the parameter $\gamma$. More precisely, for $\gamma = 1$ we obtain a membrane equation, while for $\gamma = 3$ the evolution of the limit displacement is given by a Kirchhoff-Love plate equation. The transmission conditions for the fluid flow at the effective interface $\Sigma$ are of the same type as those obtained in \cite{gahn2025effective}, but now contain additional contributions involving the limit displacement. More precisely, we obtain again the continuity of the normal velocity, and Navier-slip-type conditions for the tangential component of the velocity on both sides of $\Sigma$, which in the case $\gamma =1$ involve the time derivative of the effective displacement, and in the case $\gamma = 3$ the time-derivative of the out-of-plane displacement. Furthermore, the jump in the normal component of the normal stress turns out to be zero for for $\gamma = 1$, which is due to the contribution of the elastic structure, and does not hold in general for the rigid case.  In both cases $\gamma = 1,3$, the effective elasticity problem is coupled to the fluid system via force terms. 

To the best of our knowledge, the effective model derived in this paper is the first rigorous homogenization result for a microscopic fluid-structure interaction model, which allows mass transfer through the membrane. In \cite{Orlik2023}, where an application of fluid flow through 
to thin filters consisting of elastic slender yarns in contact 
was considered, the authors extended the effective model from \cite{gahn2022derivation} by incorporating an additional (phenomenological) interface condition obeying Darcy's law, in order to include mass transport through the filter. Comparing the model from \cite{Orlik2023} with the model derived in this paper, some similarities can be observed with the case $\gamma=3$, e.g., with regard to the transmission condition for the jump of the normal component of the normal fluid stress. However, there are also significant differences, such as the fact that in \cite{Orlik2023} all components of the fluid velocity are continuous across the interface $\Sigma$ and there is no Navier-slip-type boundary condition on each side of the interface. There, only a condition for the jump of the tangential component of the normal fluid stress is assumed. 

For the derivation of the effective model, we use two-scale convergence for thin domains \cite{NeussJaeger_EffectiveTransmission} and for rapidly oscillating surfaces \cite{bhattacharya2022homogenization}, which allow to deal simultaneously with the homogenization of the periodic structures in the membrane and the reduction of the thin membrane to a (lower dimensional) interface. In a first step, we provide \textit{a priori} estimates for the solutions to the microscopic problems with an explicit dependence on the scale parameter $\vareps$. To obtain refined estimates for the velocity in the layer and for the displacement, we first control the velocity in the layer via the bulk velocities, and afterwards, we estimate the displacement via the continuity of the velocities across $\Gamma_\vareps$. Based on the \textit{a priori} estimates, we prove (two-scale)  compactness, where in case of the displacement we have to distinguish between the cases $\gamma =1$ and $\gamma =3$. In the latter case, in the two-scale limit a Kirchhoff-Love-displacement is obtained, which was already shown in \cite{GahnJaegerTwoScaleTools}, see  also \cite{griso2020homogenization} for the pure elastic case for a thin perforated layer. For the case $\gamma =1$, where all components of the displacement and of the symmetric gradient are of the same order (and this is not valid for the gradient itself), we prove a new compactness result for the symmetric gradient. To pass to the limit in the microscopic model, we again have to deal separately with the two cases. The derivation of the limit problem for the flow equations requires a careful choice of test functions adapted to the properties of the limit velocities and to the scaling of the elastic stress term. Compared to \cite{gahn2025effective}, where the microscopic fluid equations had a similar scaling, however no-slip boundary conditions were assumed on $\Gamma_\vareps$, here due to the coupling to the elastic structure, in the case $\gamma=1$, also test-functions which are constant on the solid domain are allowed. A novel aspect compared to \cite{gahn2022derivation}, especially regarding the derivation of effective elasticity equations, is the representation of the test functions by means of solutions to cell problems. This allows us to introduce effective coefficients and to derive effective transmission conditions on $\Sigma$, involving both the velocity and the displacement. 

In summary, the novel contributions of our paper are:
\begin{itemize}
    \item $\vareps$-uniform \textit{a priori} estimates for a fluid-structure interaction problem in a thin layer coupled to fluid bulk domains 
    \item New two-scale compactness result for the symmetric gradient for the case when all components of the displacement and of the symmetric gradient are of the same order (and this is not valid for the gradient itself)
    \item Derivation of effective models based on the systematic construction of test-functions with the help of solutions to cell problems
    \item Rigorous derivation of effective interface laws of Navier-slip-type involving the elastic displacement
    \item Identification of a proper scaling for a microscopic fluid-structure interaction model that enables an effective mass transport through the membrane
\end{itemize}
We have already mentioned previous contributions related to our paper.  Further, there is a large literature on the derivation of effective interface conditions for fluid flow through rigid sieves and filters. Here, we have to mention  the works \cite{AllaireII1991,BourgeatGipoulouxMarusicPaloka2001,ConcaI1987,ConcaII1987,sanchez1982boundary} and refer to \cite{gahn2025effective} for a more complete list of references and details, in particular including the treatment of Stokes-flow through thin porous
layers not coupled to bulk regions. Concerning the fluid flow through thin poroelastic layers (not coupled to bulk domains), two contributions more closely related to our work are \cite{buvzanvcic2024poroelastic} and \cite{gahn2024derivation}, which deal with the derivation of Biot-type plate models. In \cite{gahn2024derivation} the thickness and the periodicity of the layer are both of order $\vareps$, while in \cite{buvzanvcic2024poroelastic} the thickness of the layer is large compared to the periodicity and there are several layers with different microstructures. Although we also consider a thin poroelastic layer, in the limit we do not obtain a Biot-type plate equation on the interface $\Sigma$. A crucial point is that the macroscopic limit pressure on $\Sigma$ (which has a main contribution to the Biot-system) is equal to $0$, due to the coupling to the bulk domains.

The paper is organized as follows. The microscopic model including the assumptions on the data is formulated in Section \ref{sec:micro_model}. In Section \ref{sec:main_results} the main results of the paper including the macroscopic model are stated. This section also aims to give a rough overview of the derivation of the effective models for the two cases $\gamma = 1,3$. In Section \ref{sec:apriori} $\vareps$-uniform \textit{a priori} estimates for the microscopic solutions are proved. Compactness results for the microscopic solutions are proved in Section \ref{sec:compactness_results}. 
In Sections \ref{sec:limit_problem_gamma1} and \ref{sec:limit_problem_gamma3}, macroscopic models including effective interface laws are derived for the fluid flow and for the elasticity problems, corresponding to the cases $\gamma = 1,3$ respectively. A discussion of the obtained effective interface laws also in relation to literature contributions is given in Section \ref{sec:conclusion}. Finally, some auxiliary results and estimates are given in the appendix A. Furthermore, in appendix B, we briefly recall the definition of two-scale convergence for thin  layers with microstructure together with known compactness results.

\section{The microscopic model}
\label{sec:micro_model}

\subsection{The microscopic geometry}
\label{sec:micro_geometry}

The geometry of the microscopic domain corresponds to that defined in \cite{gahn2025effective}. The difference is that we only consider the case when the solid part of the thin layer does not touch the bulk domains. However, for the sake of completeness, we repeat the description of the microscopic geometry from \cite{gahn2025effective}, with the appropriate adjustments.
We consider the domain $\Omega := \Sigma \times (-H, H) \subset \R^3$ with $H > 0$, and 
 $\Sigma = (a,b)\subset \R^2$ with $a,b \in \Z^2$ and $a_i < b_i$ for $i=1,2$. Further, we assume that $\vareps^{-1} \in \N$. The domain $\Omega$  consists of two bulk domains 
\begin{align*}
\oeps^+ := \Sigma \times (\vareps,H),  \qquad \oeps^- := \Sigma \times (-H,-\vareps),
\end{align*}
which are separated by the thin layer
\begin{align*}
\oem := \Sigma \times (-\vareps,\vareps).
\end{align*}
Within the  thin layer we have a fluid part $\oemf$ and a solid part $\oems$, which have a periodical microscopic structure. More precisely, we define  the reference cell
\begin{align*}
Z := Y\times (-1,1) := (0,1)^2 \times (-1,1),
\end{align*}
with top and bottom
\begin{align*}
S^{\pm} := Y \times \{\pm 1\}.
\end{align*}
For $n\in \N$, let us denote the interior of a set $M \subset \mathbb{R}^n$ by $\mathrm{int}(M)$.
The cell $Z$ consists of a solid part $Z_s\subset Z $, see Figure \ref{fig:FigureMicroDomain}, and a fluid part $Z_f \subset Z$ with common interface $\Gamma = \mathrm{int}\left( \overline{Z_s} \cap \overline{Z_f} \right)$. Hence, we have 
\begin{align*}
Z = Z_f \cup Z_s \cup \Gamma.
\end{align*}
In this paper, we consider the case when $Z_s$ does not touch the boundary parts $S^\pm$ of $Z$, that is,
\begin{align}\label{ZscapSpm}
\overline{Z_s} \cap S^\pm = \emptyset.
\end{align}
Furthermore, we assume that $Z_f$ and $Z_s$ are open, connected   with Lipschitz-boundary, and the lateral boundary is $Y$-periodic which means that for $i=1,2$ and $\ast \in \{s,f\}$
\begin{align*}
\left(\partial Z_{\ast} \cap \{y_i = 0\}\right) + e_i = \partial Z_{\ast} \cap \{y_i=1\}.
\end{align*}
The outer unit normal with respect to $Z_{\ast}$ for $\ast \in \{s,f\}$ is denoted by $\nu_{\ast}$.
We introduce the set $K_{\vareps}:= \{k \in \Z^2 \times \{0\} \, : \, \vareps(Z + k) \subset \oem\}$.  Clearly, we have $$\oem = \mathrm{int}\left(\bigcup_{k\in K_{\vareps}} \vareps(\overline{Z} + k)\right).
$$
Now, we define the fluid and solid part of the thin layer, see Figure \ref{fig:FigureMicroDomain}, by
\begin{align*}
\oemf := \mathrm{int} \left( \bigcup_{k \in K_{\vareps}} \vareps \left(\overline{Z_f} + k \right) \right),
\quad 
\oems := \mathrm{int} \left( \bigcup_{k \in K_{\vareps}} \vareps \left(\overline{Z_s} + k \right) \right).
\end{align*}
The fluid-solid interface between the fluid and the solid domain in the membrane is denoted by 
\begin{align*}
\geps := \mathrm{int}\left( \overline{\oems} \cap \overline{\oemf}\right).
\end{align*}
The interfaces between the fluid phase $\oemf$ and the bulk domains $\oeps^{\pm}$ are defined by
\begin{align*}
\sepm := \Sigma \times \{\pm \vareps\}.
\end{align*}
Note that, due to  \eqref{ZscapSpm}, the solid phase $\oems$ does not touch the bulk domains.
Altogether, we have the following decomposition of the domain $\Omega$
\begin{align*}
\Omega &= \oeps^+ \cup \oeps^- \cup \oem \cup  S_{\vareps}^+ \cup S_{\vareps}^-
\\
&= \oeps^+ \cup \oeps^- \cup \oems \cup \oemf \cup \geps \cup  S_{\vareps}^+ \cup S_{\vareps}^-.
\end{align*}
The whole fluid part is defined by
\begin{align*}
\oef := \Omega \setminus \overline{\oems}.
\end{align*}
\begin{figure}
\hspace{1cm}\begin{minipage}{4.6cm}
\centering
\includegraphics[scale=0.1025]{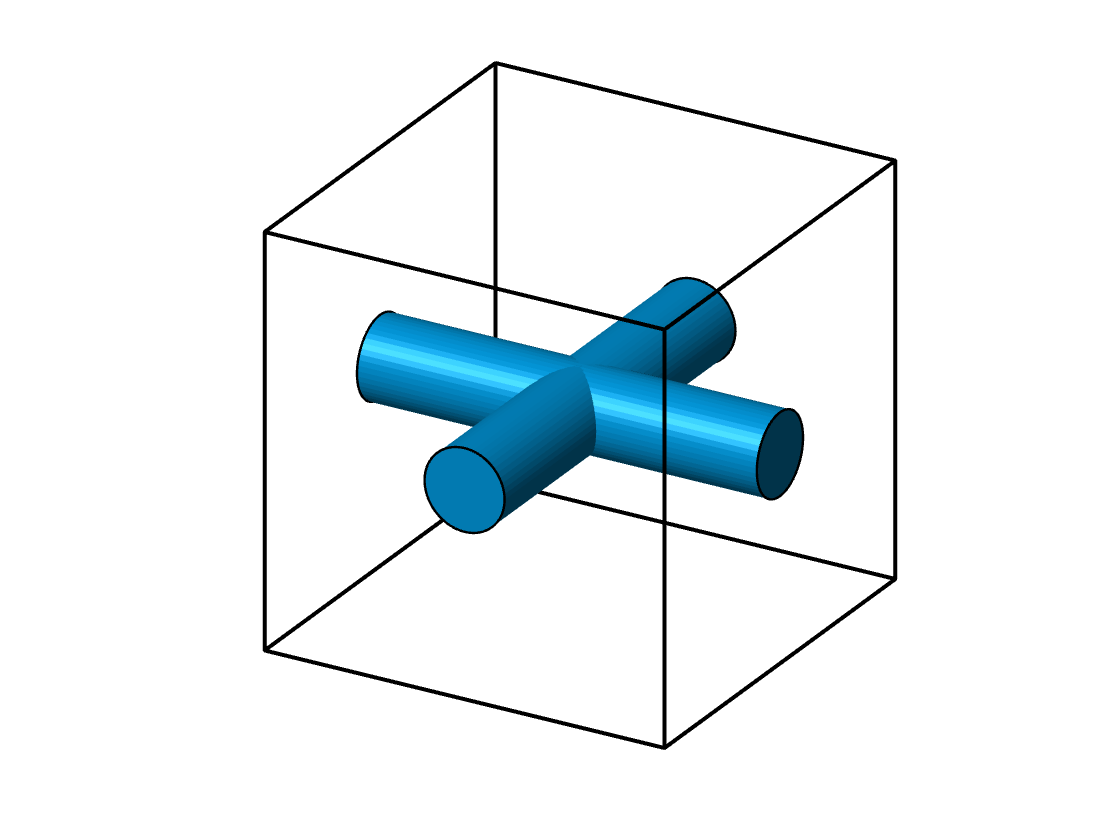}
\end{minipage} 
\hspace{0.3cm}
\begin{minipage}{4.7cm}
\centering
\includegraphics[scale=0.1725]{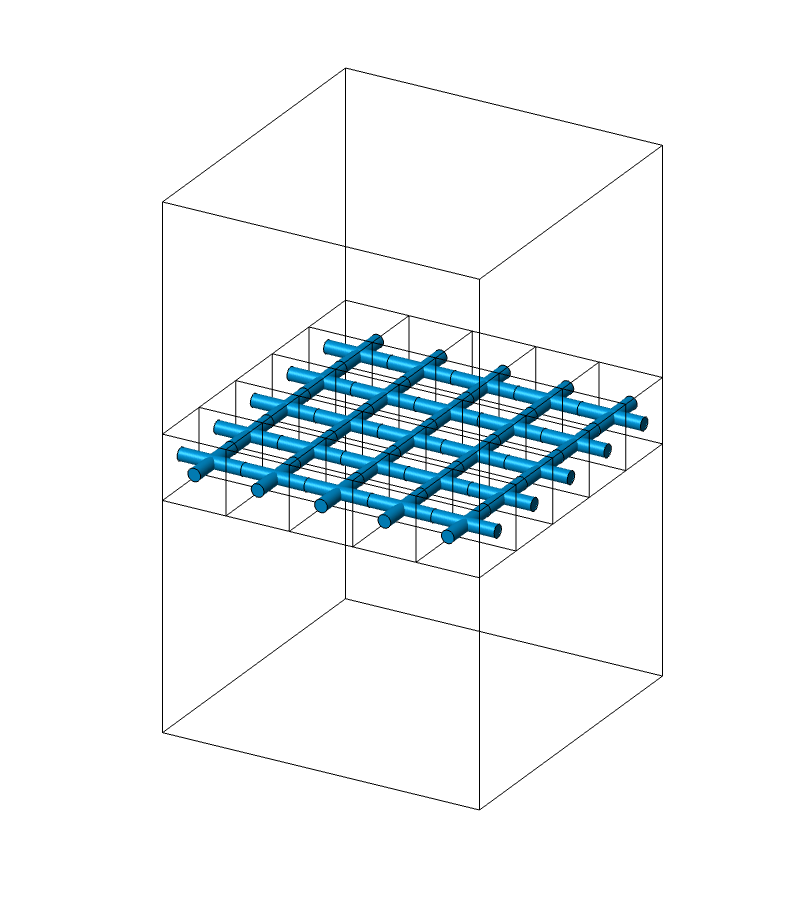}
\end{minipage}
\caption{Left: A reference cell $Z$ for the porous layer with the solid part $Z_s$ highlighted by the coloring. 
Right: The microscopic domain $\Omega$ with the porous layer $\oem$ consisting of the fluid part $\oemf$ and the solid part $\oems$.}
\label{fig:FigureMicroDomain}
\end{figure}
We further assume that the domains $\oef$, $\oemf$, and $\oems$ are connected and Lipschitz. 

Corresponding to the microscopic geometry, we split the boundary $\partial \Omega$ into several parts ($\ast \in \{f,s\}$).
The upper and lower boundary of $\Omega$ is denoted by
\begin{align*}
    \partial_N \Omega &:= \bigcup_{\pm} \partial_N \Omega^\pm:= \bigcup_{\pm}\Sigma \times \{\pm H\}.
\end{align*}
Here a stress boundary condition is assumed for the fluid flow. 
Furthermore, we have the boundary portions where Dirichlet boundary conditions are imposed:
\begin{align*}
\partial_D \oeps^+ &:= \partial \Sigma \times (\vareps,H),
\\
\partial_D \oeps^- &:= \partial \Sigma \times (-H,-\vareps),
\\
\partial_D \oem &:= \partial \Sigma \times (-\vareps,\vareps),
\\
\partial_D \oeps^{M,\ast} &:= \mathrm{int} \left(\partial_D \oem \cap \partial \oeps^{M,\ast}\right),
\\
\partial_D \oef &:= \partial_D \oemf \cup \bigcup_{\pm} \partial_D \oeps^{\pm}.
\end{align*}

In the limit $\vareps \to 0 $ the thin layer $\oem$ is reduced to the interface $\Sigma$ and the domains $\oeps^{\pm}$ converge to the  macroscopic bulk domains $\Omega^{\pm}$ defined by
\begin{align*}
\Omega^+ := \Sigma \times (0,H),
\quad 
\Omega^- := \Sigma \times (-H,0),
\end{align*}
and the Dirichlet-parts of the boundary are given by
\begin{align*}
\partial_D \Omega^+ := \partial \Sigma \times (0,H), 
\quad 
\partial_D \Omega^- := \partial \Sigma \times (-H,0), 
\end{align*}
and
\begin{align*}
\partial_D \Omega := \partial \Sigma \times (-H,H).
\end{align*}
The outer unit normal to $\Omega^\pm$ at $\Sigma$ is given by $\nu^\pm=\mp e_3$, where $e_3$ is the third standard unit vector in $\R^3$.

\subsection{Notations } 

For $U\subset \R^m$ open with $m \in \N$ we denote by $L^p(U)$ with $p \in [1,\infty]$ the usual Lebesgue spaces, and for $m \in \N_0$ the Sobolev space is defined by $W^{1,p}(U)$. For $p=2$ we shortly write $H^1(U):= W^{1,2}(U)$. For $\beta \in (0,1)$ we denote by $H^{\beta}(U)$ the Sobolev-Slobodeckii space (for $p=2$). For norms of vector valued functions with values in $\R^n$ for $n\in \N$, we usually skip the exponent, for example we write $\|\cdot\|_{L^2(U)} := \|\cdot\|_{L^2(U)^n}$. If $U$ is bounded, we define
\begin{align}
    L_0^2(U) := \left\{ \phi \in L^2(U)\, : \, \int_U \phi dx = 0\right\}.
\end{align}

We use the notation $\bar{x} = (x_1, x_2) \in \Sigma$ for a vector $x= (x_1, x_2, x_3)\in \mathbb{R}^3$. For an arbitrary function $\phieps : \oef \rightarrow \R^m$ for $m\in \N$ we define the restrictions to the bulk domains and the fluid part of the membrane by
\begin{align*}
\phieps^{\pm} := \phieps\vert_{\oeps^{\pm}}, \qquad \phieps^M := \phieps\vert_{\oemf}.
\end{align*}
The characteristic function of a set $A$ is denoted by $\chi_A$. 
For function spaces we use the index $\#$ to indicate functions which are periodic with respect to the first two variables. 
More precisely, we have
\begin{align*}
C_{\#}^{\infty}(\overline{Z}):= \left\{ \phi \in C^{\infty}(\R^2 \times [-1,1] ) \, : \, \phi \mbox{ is } Y\mbox{-periodic}\right\}.
\end{align*} 
Here, a $Y$-periodic function $\phi \in C^{\infty}(\R^2 \times [-1,1] )\to \R$ fulfills $\phi(y + e_i) = \phi(y)$ for all $y\in \R^2 \times [-1,1]$, and $i=1,2$, where $e_i$ are standard unit vectors in $\R^3$.
We denote by $H_{\#}^1(Z)$ the  closure of $C_{\#}^{\infty}(\overline{Z})$ with respect to the $H^1(Z)$-norm, and by $H^1_{\#}(Z_{\ast})$ for $\ast \in \{s,f\}$ the restriction of $H^1_{\#}(Z)$-functions on $Z_{\ast}$. We emphasize that for $L^p$-spaces we avoid to write $L^p_{\#}$, since these functions have no traces and if not stated otherwise we extend such functions periodically with respect to the first two variables.

For a Lipschitz domain $U\subset \R^n$ with $n\in \N$ and $\omega \subset \partial U$, we define
\begin{align*}
H^1(U,\omega):= \left\{ u \in H^1(U)\, :\, u= 0 \mbox{ on } \omega\right\}.
\end{align*}

For a Banach space $B$ its dual is given by $B^{\ast}$ and we denote the duality pairing between the dual space $B^{\ast}$ and $B$ by $\langle \cdot ,\cdot \rangle_B$. Further, for an arbitrary open set $U\subset \R^m$ with $m\in \N$ we write $L^p(U,B)$ for the usual Bochner spaces (for $p\in [1,\infty]$).

We denote the symmetric gradient of a weakly differentiable vector field $u:\Omega \to \mathbb{R}^3$ by
$D(u):= \frac12 \left(\nabla u + \nabla u^T\right)$. Furthermore, for a weakly differentiable vector field $u:\Sigma \to \mathbb{R}^2$, we use the notation $D_{\x}u$ for a matrix in $\mathbb{R}^{2\times 2}$ as well as for the trivial embedding in $\mathbb{R}^{3\times 3}$ via 
$\left( \begin{array}{ccc} \partial_1 u_1 &\frac{1}{2}(\partial_1 u_2 + \partial_2 u_1) & 0\\ \frac{1}{2}(\partial_1 u_2 + \partial_2 u_1) &\partial_2 u_2 & 0\\ 0 & 0 & 0 \end{array} \right)$.

\subsection{The microscopic problem}
Now, we formulate the microscopic model together with the assumptions on the data and give the definition of a weak solution of the problem.

In the fluid part $\oef$ we have the fluid velocity $\veps = (\veps^+,\veps^M,\veps^-): (0,T)\times \oef \rightarrow \R^3$ and the fluid pressure $\peps = (\peps^+,\peps^M,\peps^-) : (0,T)\times  \oef \rightarrow \R$. 
The evolution of the  velocity and pressure of the fluid is given by
\begin{subequations}\label{MicroscopicModel}
\begin{align}
\label{def:micro_equations_fluid_pde}
\partial_t \veps^{\pm} - \nabla \cdot D( \veps^{\pm}) + \nabla \peps^{\pm}  &= f_{\vareps}^{\pm} &\mbox{ in }& (0,T)\times \oeps^{\pm},
\\
 \partial_t \veps^M- \vareps\nabla \cdot D( \veps^M )+  \foe\nabla \peps^M  &= 0 &\mbox{ in }& (0,T)\times \oemf,
\\
\nabla \cdot \veps &= 0 &\mbox{ in }& (0,T)\times \oef,
\\
\left(-\peps I + D(\veps) \right) \nu &= 0 &\mbox{ on }& (0,T) \times \partial_N \Omega,
\\
\label{MicroModelBCDirichlet}
\veps &= 0 &\mbox{ on }& (0,T) \times \partial_D \oef,
\\
\label{ContinuityVelocity}
\veps^\pm & = \veps^M & \mbox{ on }& (0,T) \times \sepm,
\\
\label{ContinuityNormalStress}
\left(-\peps^\pm I + D(\veps^\pm) \right) \nu^\pm = \left(-\frac{1}{\vareps}\peps^M I + \vareps D(\veps^M) \right)& \nu^\pm & \mbox{ on }& (0,T) \times \sepm,
\\
\veps(0) &= 0 &\mbox{ in }& \oef.
\end{align}
Here,  
$f_{\vareps}^\pm$ are the bulk forces and $\nu$ is the outer unit normal of $\oef$. At the interface $\sepm$ we assume the natural transmission conditions \eqref{ContinuityVelocity}-\eqref{ContinuityNormalStress} describing the continuity of the velocity and of the normal stresses. Here $\nu^\pm = \mp e_3$ is the outer unit normal to $\oeps^\pm$ on $\sepm$.

The evolution of the displacement $\ueps:(0,T)\times \oems \rightarrow \R^3$ is described by 
\begin{align}
    \partial_{tt} \ueps - \frac{1}{\vareps^{\gamma}} \nabla \cdot (A_{\vareps} D(\ueps)) &= 0 &\mbox{ in }& (0,T)\times \oems,
    \\
    \ueps &= 0 &\mbox{ on }& (0,T)\times \partial_D \oems,
    \\
    \ueps(0) = \partial_t \ueps(0) &=0 &\mbox{ in }& \oems,
\end{align}
with $\gamma \in \{1,3\}$. 
At the interface $\geps$ between the solid and the fluid we assume continuity of the velocities and of the normal stresses:
\begin{align}
    \veps^M &= \partial_t \ueps &\mbox{ on }& (0,T)\times \geps,
    \\
    \left(-\frac{1}{\vareps}\peps^M I + \vareps D(\veps^M) \right) \nu &= -\frac{1}{\vareps^{\gamma}} A_{\vareps} D(\ueps)\nu &\mbox{ on }& (0,T)\times \geps.
\end{align}
\end{subequations}

\begin{remark}\
\begin{enumerate}
[label = (\roman*)]
\item We only consider non-trivial forces in the bulk regions $\oeps^{\pm}$. The idea is that we  want to investigate the influence of such forces on flow and displacement within the membrane. For applications, these forces seem to be the most important ones (together with pressure boundary conditions on $\partial_N \Omega$, which can be easily included in the model). However, the following methods can be extended to more general forces in the equations for flow and structure in the membrane.
\item In this paper, for the sake of simplicity, we only consider domains in $\R^3$, which play the most important role in the applications. However, all results can be extended to arbitrary spatial dimensions in an obvious way.

\end{enumerate}
\end{remark}

\noindent\textbf{Assumptions on the data:}
\begin{enumerate}[label = (A\arabic*)]
\item \label{ass:rhs_feps_pm} For the bulk forces we assume $f_{\vareps}^\pm \in L^2(\oeps^\pm)$ and there exists $f_0^\pm \in L^2(\Omega^{\pm})$ such that $\chi_{\oeps^{\pm}} f_{\vareps}^\pm \rightharpoonup f_0^\pm $ in $L^2(\Omega^\pm)$. In particular we have that $\chi_{\oeps^{\pm}} f_{\vareps}^\pm$ is bounded in $L^2(\Omega^\pm)$.

\item\label{AssumptionElasticityTensor} The elasticity tensor $A_{\epsilon}$ is defined by $A_{\epsilon}(x) := A\left(\fxe\right)$ with $A\in L_{\#}^{\infty}(Z^s)^{3\times 3 \times 3 \times 3}$ symmetric and coercive on the space of symmetric matrices, more precisely for $i,j,k,l=1,2,3$
\begin{align*}
A_{ijkl} &= A_{jilk} = A_{ljik},
\\
A(y) B : B &\geq c_0 \vert B\vert^2 \quad \mbox{ for almost  every } y \in Z,
\end{align*}
with $c_0 > 0$ and all $B \in \R^{3\times 3}$ symmetric.

\end{enumerate}

\begin{definition}[Weak formulation of the microscopic model]\label{def:weak_solution_micro_model}
Let $\gamma = 1,3$. We say that the triple $(\veps,\peps,\ueps)$ with $\veps = (\veps^+,\veps^M, \veps^-)$ and $\peps = (\peps^+ , \peps^M , \peps^-)$ is a weak solution of the microscopic problem $\eqref{MicroscopicModel}$ if 
\begin{align*}
    \veps &\in L^2((0,T),H^1(\oef,\partial_D \oef))^3 \cap H^1((0,T),L^2(\oef))^3,
    \\
    \peps &\in L^2((0,T)\times \oef),
    \\
    \ueps &\in H^1((0,T),H^1(\oems,\partial_D \oems ))^3 \cap H^2((0,T),L^2(\oems))^3
\end{align*}
with $\nabla \cdot \veps = 0$ and $\veps^M = \partial_t \ueps $ on $\geps$ and for all $\phieps \in H^1(\Omega, \partial_D\Omega)^3$ it holds almost everywhere in $(0,T)$ 
\begin{align}
\begin{aligned}
\label{eq:Var_Micro_veps}
\sum_{\pm}& \left\{ \int_{\oeps^{\pm}} \partial_t \veps^{\pm} \cdot \phieps dx +  \int_{\oeps^{\pm}}  D( \veps^{\pm}) : D( \phieps) dx - \int_{\oeps^{\pm}} \peps^{\pm} \nabla \cdot \phieps dx\right\}
+  \int_{\oemf} \partial_t \veps^M \cdot \phieps dx 
\\
+& \vareps \int_{\oemf} D( \veps^M) : D( \phieps) dx 
 -\foe  \int_{\oemf} \peps^M \nabla \cdot \phieps dx 
\\
+& \frac{1}{\vareps^{\gamma}} \int_{\oems} A_{\vareps} D(\ueps): D(\phieps) dx 
=\sum_{\pm} \int_{\oeps^{\pm}} f_{\vareps}^{\pm} \cdot \phieps dx.
\end{aligned}
\end{align}
Further, the initial conditions $\veps(0) = 0$ and $\ueps(0) = \partial_t \ueps (0) = 0$ hold.
\end{definition}

\section{Main results}
\label{sec:main_results}

In the following, we give the macroscopic model for the cases $\gamma = 1$ and $\gamma = 3$ and formulate the main results of the paper. We will show that the microscopic solutions $(\veps,\peps)$ and $\ueps$  converge in a suitable sense to limit functions $(v_0,p_0)$ and $(u_0,u_1)$, respectively. For the fluid velocity and the fluid pressure $(v_0,p_0)$ we obtain in both cases $\gamma = 1$ and $\gamma = 3$ an incompressible Stokes-system in the limit, with effective interface conditions across $\Sigma$. These are the continuity of the normal velocity, a Navier-slip-type boundary condition for the tangential normal stress, and a jump condition for the normal component of the normal stress. Furthermore, for $\gamma = 1$ it turns out that the normal component of the normal stress is continuous.
\\
The macroscopic system for the displacement depends highly on the choice of $\gamma$. More precisely, for $\gamma = 1$ we obtain a membrane equation, while for $\gamma = 3$ the evolution of the limit displacement is given by a Kirchhoff-Love plate equation. In both cases the problem is coupled to the fluid system via force terms.

\subsection{The case $\gamma  = 1$}

We start with the formulation of the macroscopic fluid problem, which is given by an instationary Stokes equation. More precisely, the limit functions $v_0^{\pm}: (0,T)\times \Omega^{\pm} \rightarrow \R^3$ and $p_0^{\pm}:(0,T)\times \Omega^{\pm} \rightarrow \R$ are solutions of the problem:
\begin{subequations}\label{def:Macro_Stokes_model_strong}
\begin{align}
\label{def:Macro_Stokes_model_PDE}
    \partial_t v_0^{\pm} - \nabla\cdot D(v_0^{\pm}) + \nabla p_0^{\pm} &= f_0^{\pm} &\mbox{ in }& (0,T)\times \Omega^{\pm},
\\
\nabla \cdot v_0^{\pm} &= 0 &\mbox{ in }& (0,T)\times \Omega^{\pm},
\\
v_0^{\pm} &= 0 &\mbox{ on }& (0,T)\times \partial_D \Omega^{\pm},
\\
[v_0^+]_3 &= [v_0^-]_3 &\mbox{ on }& (0,T)\times \Sigma,
\\
-(D(v_0^{\pm}) - p_0^{\pm}I)\nu &= 0 &\mbox{ on }& (0,T)\times \partial_N\Omega^\pm,
\\
\label{Macro_Model_Fluid_Stress_normal}
-\llbracket (D (v_0) - p_0 I ) \nu\cdot \nu\rrbracket &=
F_1(\partial_t u_0,v_0^+,v_0^-)
&\mbox{ on }& (0,T)\times \Sigma,
\\
\label{Macro_Model_Fluid_Stress_tangential}
- [(D(v_0^{\pm}) - p_0^{\pm}I)\nu^{\pm}]_t &= [L^{\pm} \partial_t u_0]_t +[K^{\pm}v_0^{\pm}]_t + M^{\mp}v_0^{\mp} &\mbox{ on }& (0,T)\times \Sigma,
\\
v_0^{\pm}(0) &= 0 &\mbox{ in }& \Omega^{\pm}.
\end{align}
\end{subequations}
where 
\begin{align} \label{RHS_gamma=1}
F_1(\partial_t u_0,v_0^+,v_0^-) := L^+ \partial_t u_0 \cdot \nu^+ - L^- \partial_t u_0 \cdot \nu^- +  K^+v_0^+ \cdot \nu^+ -K^-v_0^- \cdot \nu^- =0,
\end{align}
for almost every $(t,x) \in (0, T) \times \Sigma$. 
Here, $\llbracket \phi\rrbracket := \phi^+ - \phi^-$ on $\Sigma$ denotes the jump across $\Sigma$ and $[\cdot]_t$ the tangential part of a vector field. The effective coefficients $K^\pm \in \R^{3\times 3}, M^\pm \in \R^{3\times 3}$, and  $L^\pm \in \R^{3\times 3}$ are defined in \eqref{Def:K}, \eqref{Def:M}, and \eqref{Def:L} respectively, in Section \ref{sec:limit_problem_gamma1}, and are given via suitable Stokes-cell problems.
We call the tuple $(v_0^\pm,p_0^\pm)$ a weak solution of $\eqref{def:Macro_Stokes_model_strong}$, if 
\begin{align*}
    v_0^{\pm} &\in L^2((0,T),H^1(\Omega^{\pm},\partial_D \Omega^{\pm}))^3 \cap H^1((0,T),L^2(\Omega^{\pm}))^3, \mbox{ and }
    p_0^{\pm} \in L^2((0,T)\times \Omega^{\pm})
\end{align*}
with $[v_0^+]_3 = [v_0^-]_3$ on $(0,T)\times \Sigma$,  $\nabla \cdot v_0^{\pm} = 0$, and for all $(\phi^+,\phi^-) \in H^1(\Omega^+,\partial_D \Omega^+)^3 \times H^1(\Omega^-,\partial_D \Omega^-)^3$ with $\phi^+_3 = \phi^-_3$ on $\Sigma$ it holds almost everywhere in $(0,T)$ 
\begin{align}
\begin{aligned}\label{eq:var_macro_fluid}
&\sum_{\pm}  \left\{ \int_{\Omega^{\pm}} \partial_t v_0^{\pm} \cdot \phi^{\pm} dx +\int_{\Omega^{\pm}} D(v_0^{\pm}) : D(\phi^{\pm}) dx -  \int_{\Omega^{\pm}} p_0^{\pm} \nabla \cdot \phi^{\pm} dx   \right\} 
\\
&+\sum_{\pm} \int_{\Sigma} [L^{\pm}\partial_t u_0]_t \cdot \phi^{\pm} d\x+ \sum_{\pm} \int_{\Sigma} [K^{\pm} v_0^{\pm}]_t \cdot \phi^{\pm} d\x + \int_{\Sigma} M^{-}v_0^- \cdot \phi^+ + M^{+}v_0^+ \cdot \phi^- d\x 
\\
&= \sum_{\pm} \int_{\Omega^{\pm}} f^{\pm}_0 \cdot \phi^{\pm} dx,
\end{aligned}
\end{align}
together with the initial condition $v_0^{\pm}(0) = 0$.

The weak equation $\eqref{eq:var_macro_fluid}$ is obtained with an elemental calculation by formally multiplying $\eqref{def:Macro_Stokes_model_PDE}$ with suitable test functions, integrating by parts, and decomposing the normal stress into its normal and tangential part and using the interface conditions in $\eqref{def:Macro_Stokes_model_strong}$. 

The macro-model for the limit displacement $u_0$ for $\gamma = 1$ reads as follows: With $\hat{u}_0= (u_0^1,u_0^2)$ we have that the limit function $u_0:(0,T)\times \Sigma \to \R^3$ is a solution of the problem:
\begin{subequations}\label{eq:macro_model_displacement_gamma1}
\begin{align}
\label{eq:macro_model_displ_gamma1_PDE}
L^{\Gamma} \partial_t u_0 - \nabla_{\x} \cdot (A^{\ast} D_{\x}(\hat{u}_0)) &=- \sum_{\pm} [L^{\pm}]^T v_0^{\pm}  &\mbox{ in }& (0,T) \times \Sigma,
\\
\hat{u}_0 &= 0 &\mbox{ on }& (0,T)\times \partial \Sigma,
\\
u_0(0) &= 0 &\mbox{ in }& \Sigma,
\end{align}
\end{subequations}
with the effective tensor $L^{\Gamma} \in \R^{3\times 3}$ defined in \eqref{Def:LGamma}, see Section \ref{sec:limit_problem_gamma1}, and given via a suitable Stokes-cell problem. Furthermore, the effective elasticity tensor $A^*$ is defined in \eqref{def:effective_elasticity_tensor} by means of suitable elasticity-cell problems.

We say that $u_0$ is a weak solution of the problem $\eqref{eq:macro_model_displacement_gamma1}$ if 
\begin{align*}
  u_0 \in H^1((0,T), L^2(\Sigma))^3 \cap L^2((0,T),H^1_0(\Sigma)^2 \times L^2(\Sigma))  
\end{align*}
and for all $\psi \in H_0^1(\Sigma)^2 \times L^2(\Sigma)$ it holds almost everywhere in $(0,T)$ that
\begin{align}\label{eq:var_macro_model_displ_gamma1}
    \int_{\Sigma} L^{\Gamma} \partial_t u_0 \cdot \psi d\x + \int_{\Sigma} A^{\ast} D_{\x}(\hat{u}_0) : D_{\x}(\hat{\psi}) d\x + \int_{\Sigma} \sum_{\pm} L^{\pm} \psi  \cdot v_0^{\pm} d\x = 0,
\end{align}
together with the initial condition $u_0(0) = 0$.

\begin{remark}\label{rem:macro_model_displ_gamma1}\
\begin{enumerate}[label = (\roman*)]
    \item The third component of $\nabla_{\x} \cdot (A^{\ast} D_{\x}(\hat{u}_0))$ vanishes and therefore the third component in the equation $\eqref{eq:macro_model_displ_gamma1_PDE}$ is an ODE.
    \item It holds that (see proof of \eqref{eq:interface_condition_normal_stress} in Section \ref{sec:limit_problem_gamma1})
    \begin{align*}
        -\nabla_{\x} \cdot (A^{\ast} D_{\x}(\hat{u}_0)) = -\llbracket (-D(v_0) + p_0 I)\nu \rrbracket \qquad\mbox{on } (0,T)\times \Sigma.
    \end{align*} 
    In particular, the jump across $\Sigma$ in the third (normal) component of the effective normal stress for the bulk fluid is zero.
\end{enumerate}
\end{remark}

The main result for the case $\gamma = 1$ reads as follows (we refer to Section \ref{sec:two_scale_convergence} in the appendix for the definition of the two-scale convergence $\rats$):

\begin{theorem}\label{thm:main_gamma1}
Let $\gamma = 1$ and let $(\veps,\peps,\ueps)$ be the weak solution of the microscopic model $\eqref{MicroscopicModel}$.  Then, there exist 
\begin{align*}
 v_0^{\pm} &\in L^2((0,T),H^1(\Omega^{\pm})^3 \cap H^1((0,T),L^2(\Omega^{\pm}))^3,
    \\
    p_0^{\pm} &\in L^2((0,T)\times \Omega^{\pm}),
    \\
    u_0 &\in H^1((0,T),H^1_0(\Sigma)^2 \times L^2(\Sigma))
\end{align*}
such that up to a subsequence 
\begin{align*}
     \chi_{\oeps^{\pm}} \veps^{\pm} &\rightarrow v_0^{\pm} &\mbox{ in }& L^2((0,T)\times \Omega^{\pm})^3,
     \\
    \chi_{\oeps^{\pm}}\peps^{\pm} &\rightharpoonup p_0^{\pm} &\mbox{ weakly in }& L^2((0,T)\times \Omega^{\pm}),
    \\
    \chi_{\oems} \ueps &\rats \chi_{Z_s} u_0.
\end{align*}
The limit function $(v_0,p_0,u_0)$ with $v_0:=(v_0^+,v_0^-)$ and $p_0 := (p_0^+,p_0^-)$ is  the unique weak solution of the macroscopic problem  $\eqref{def:Macro_Stokes_model_strong}  \, + \, \eqref{eq:macro_model_displacement_gamma1}$.
\end{theorem}
The proof of this result is given in Section \ref{sec:compactness_results} and \ref{sec:limit_problem_gamma1}. More precisely, see Section \ref{sec:compactness_results} for the compactness results and additional convergence results, also including traces, time- and space-derivatives, as well as additional corrector functions. The macroscopic model is derived in Section \ref{sec:limit_problem_gamma1}.

\subsection{The case $\gamma  = 3$}

Here, the limit fluid model has the same structure as in the case $\gamma  =1$. The macroscopic equation for the displacement is now given by a Kirchhoff-Love plate equation, where the in-plane displacement $\hat{u}_1 = (u_1^1,u_1^2)$ is given as the solution of a second order equation, and the out-of-plane displacement $u_0^3$ by a forth order equation. First, we formulate the fluid problem: The limit functions $(v_0^{\pm},p_0^{\pm})$ solve again a Stokes model, where the only difference compared to the case $\gamma = 1$ lies in the force terms for the effective stress conditions across the interface $\Sigma$, depending only on the out-of-plane displacement $u_0^3$ (and the velocities $v_0^{\pm}$ on $\Sigma$).   We have
\begin{subequations}\label{def:Macro_Stokes_model_strong_gamma3}
\begin{align}
\label{def:Macro_Stokes_model_PDE_gamma3}
    \partial_t v_0^{\pm} - \nabla\cdot D(v_0^{\pm}) + \nabla p_0^{\pm} &= f_0^{\pm} &\mbox{ in }& (0,T)\times \Omega^{\pm},
\\
\nabla \cdot v_0^{\pm} &= 0 &\mbox{ in }& (0,T)\times \Omega^{\pm},
\\
v_0^{\pm} &= 0 &\mbox{ on }& (0,T)\times \partial_D \Omega^{\pm},
\\
[v_0^+]_3 &= [v_0^-]_3 &\mbox{ on }& (0,T)\times \Sigma,
\\
-(D(v_0^{\pm}) - p_0^{\pm}I)\nu &= 0 &\mbox{ on }& (0,T)\times \partial_N\Omega^\pm,
\\
\label{Macro_Model_Fluid_Stress_normal_gamma3}
-\llbracket (D (v_0) - p_0 I ) \nu\cdot \nu\rrbracket &= F_3(\partial_t u_0^3, v_0^+, v_0^-)
&\mbox{ on }& (0,T)\times \Sigma,
\\
\label{Macro_Model_Fluid_Stress_tangential_gamma3}
- [(D(v_0^{\pm}) - p_0^{\pm}I)\nu^{\pm}]_t &= \partial_t u_0^3[L^{\pm} e_3]_t +[K^{\pm}v_0^{\pm}]_t + M^{\mp}v_0^{\mp} &\mbox{ on }& (0,T)\times \Sigma,
\\
v_0^{\pm}(0) &= 0 &\mbox{ in }& \Omega^{\pm},
\end{align}
\end{subequations}
with 
\begin{align} \label{Def:RHS_gamma=3}
F_3(\partial_t u_0^3, v_0^+, v_0^-)= -L^{\Gamma}_{33} \partial_t u_0^3 + K^+ v_0^+ \cdot \nu^+  - K^- v_0^- \cdot \nu^-,
\end{align}
for almost every $(t,x) \in (0, T) \times \Sigma$. For the definition of the effective coefficients $L^{\pm}$, $K^{\pm}$, $M^{\pm}$ and $L^{\Gamma}$ we refer again to Section \ref{sec:limit_problem_gamma1}. The weak formulation for this problem is the following: We call the tuple $(v_0^\pm,p_0^\pm)$ a weak solution of \eqref{def:Macro_Stokes_model_strong_gamma3}, if  
\begin{align*}
    v_0^{\pm} &\in L^2((0,T),H^1(\Omega^{\pm},\partial_D \Omega^{\pm}))^3 \cap H^1((0,T),L^2(\Omega^{\pm}))^3, \mbox{ and }
    p_0^{\pm} \in L^2((0,T)\times \Omega^{\pm})
\end{align*}
with $[v_0^+]_3 = [v_0^-]_3$ on $(0,T)\times \Sigma$,  $\nabla \cdot v_0^{\pm} = 0$, and for all $(\phi^+,\phi^-) \in H^1(\Omega^+,\partial_D \Omega^+)^3 \times H^1(\Omega^-,\partial_D \Omega^-)^3$ with $\phi^+_3 = \phi^-_3$ on $\Sigma$ it holds almost everywhere in $(0,T)$ 
\begin{align}
\begin{aligned}\label{eq:var_macro_fluid_gamma3}
\sum_{\pm} & \left\{ \int_{\Omega^{\pm}} \partial_t v_0^{\pm} \cdot \phi^{\pm} dx +\int_{\Omega^{\pm}} D(v_0^{\pm}) : D(\phi^{\pm}) dx -  \int_{\Omega^{\pm}} p_0^{\pm} \nabla \cdot \phi^{\pm} dx   \right\} 
\\
& +\sum_{\pm}\int_\Sigma L^{\pm} \partial_t u_0 \cdot \phi^{\pm} d\x + \sum_{\pm} \int_{\Sigma} K^{\pm} v_0^{\pm} \cdot \phi^{\pm} d\x + \int_{\Sigma} M^{-}v_0^- \cdot \phi^+ + M^{+}v_0^+ \cdot \phi^- d\x \\
&=  \sum_{\pm} \int_{\Omega^{\pm}} f^{\pm}_0 \cdot \phi^{\pm} dx  
\end{aligned}
\end{align}
together with the initial condition $v_0^{\pm}(0) = 0$ (see also Corollary \ref{cor:LGamma_u_0_Lpm_v_0} for a relation between $L^{\pm}$ and $L^{\Gamma}$).
\\

Next, we formulate the Kirchhoff-Love plate equation for the displacements $u_0^3$ and $\hat{u}_1=(u_1^1,u_1^2)$:
\begin{subequations}\label{eq:macro_model_displacement_gamma3}
\begin{align}
-\nabla_{\x} \cdot \left(a^{\ast} D_{\x}(\hat{u}_1) + b^{\ast} \nabla_{\x}^2 u_0^3 \right) &= 0 &\mbox{ in }& (0,T)\times \Sigma,
\\
 \nabla_{\x}^2 : \left(b^{\ast} D_{\x}(\hat{u}_1) + c^{\ast}\nabla_{\x}^2 u_0^3 \right) &=  -L^{\Gamma}_{33} \partial_t u_0^3 + K^+ v_0^+ \cdot \nu^+  - K^- v_0^- \cdot \nu^- &\mbox{ in }& (0,T)\times \Sigma,
\\
u_0^3 = \nabla_{\x} u_0^3 \cdot \nu & = 0&\mbox{ on }& (0,T) \times \partial \Sigma,
\\
\hat{u}_1 &= 0 &\mbox{ on }& (0,T)\times \partial \Sigma,
\end{align}
\end{subequations}
The weak equation reads as follows: We say that $u_0^3$ and $\hat{u}_1$ is a weak solution of the macroscopic problem $\eqref{eq:macro_model_displacement_gamma3}$ if 
\begin{align*}
    u_0^3 \in L^2((0,T),H_0^2(\Sigma)), \qquad \hat{u}_1 \in L^2((0,T),H_0^1(\Sigma))^2,
\end{align*}
and for all $\psi_3 \in H_0^2(\Sigma)$ and $\hat{U}:=(U_1,U_2) \in H_0^1(\Sigma)^2$ it holds almost everywhere in $(0,T)$ that
\begin{align*}
 \int_{\Sigma} a^{\ast} D_{\x}(\hat{u}_1) : D_{\x}(\hat{U}) + b^{\ast} &\nabla_{\x}^2 u_0^3 : D_{\x}(\hat{U}) + b^{\ast} D_{\x}(\hat{u}_1) : \nabla_{\x}^2 \psi_3 + c^{\ast} \nabla_{\x}^2 u_0^3 : \nabla_{\x}^2 \psi_3 d\x 
 \\
 &= \int_{\Sigma} - L^{\Gamma}_{33} \partial_t u_0^3 \psi_3 + \left[K^+ v_0^+ \cdot \nu^+  - K^- v_0^- \cdot \nu^- \right] \psi_3 d\x.
\end{align*}

\begin{remark}
Comparing the macroscopic equation for $(v_0^{\pm},p_0^{\pm})$ and $(u_1,u_0^3)$, we obtain that the jump across $\Sigma$ of the normal component of the effective normal stress in the fluid is given by
\begin{align*}
\left\llbracket  (- D(v_0) + p_0 I )\nu \cdot \nu \right\rrbracket &=    -L^{\Gamma}_{33} \partial_t u_0^3 + K^+ v_0^+ \cdot \nu^+  - K^- v_0^- \cdot \nu^-
\\
&= \nabla_{\x}^2 : \left(b^{\ast} D_{\x}(\hat{u}_1) + c^{\ast}\nabla_{\x}^2 u_0^3 \right) .
\end{align*}
\end{remark}

Now, we formulate the main result for the case $\gamma = 3$:

\begin{theorem}\label{thm:main_gamma3}
Let $\gamma = 3$ and let $(\veps,\peps,\ueps)$ be the weak solution of the microscopic problem $\eqref{MicroscopicModel}$.  Then, there exist 
\begin{align*}
 v_0^{\pm} &\in L^2((0,T),H^1(\Omega^{\pm})^3 \cap H^1((0,T),L^2(\Omega^{\pm}))^3,
    \\
    p_0^{\pm} &\in L^2((0,T)\times \Omega^{\pm}),
    \\
     u_0^3 &\in H^1((0,T),H^2_0(\Sigma)),
    \\
    \hat{u}_1 &\in H^1((0,T),H^1_0(\Sigma))^2,
\end{align*}
such that up to a subsequence 
\begin{align*}
\chi_{\oeps^{\pm}} \veps^{\pm} &\rightarrow v_0^{\pm} &\mbox{ in }& L^2((0,T)\times \Omega^{\pm})^3,
     \\
    \chi_{\oeps^{\pm}}\peps^{\pm} &\rightharpoonup p_0^{\pm} &\mbox{ weakly in }& L^2((0,T)\times \Omega^{\pm}),
    \\
     \chi_{\oems}\ueps^3 &\rats \chi_{Z_s} u_0^3,
\\
\chi_{\oems}\frac{\ueps^{\alpha}}{\epsilon} &\rats \chi_{Z_s}\big(u_1^{\alpha} - y_3 \partial_{\alpha} u_0^3\big),
\end{align*}
where $\hat{u}_1 = (u_1^1, u_1^2)$ and $\alpha =1,2$.
The limit function $(v_0,p_0,u_0^3,\hat{u}_1)$ with $v_0:=(v_0^+,v_0^-)$ and $p_0 := (p_0^+,p_0^-)$ is  the unique weak solution of the macroscopic problem  $\eqref{def:Macro_Stokes_model_strong_gamma3}  \, + \, \eqref{eq:macro_model_displacement_gamma3}$.
\end{theorem}
The proof of the compactness results are given in Section \ref{sec:compactness_results}, in particular including the convergence of the symmetric gradient of $\ueps$. The proof that $(v_0,p_0,u_0^3,\hat{u}_1)$ is a solution of the macroscopic model coupling fluid flow and plate equation can be found in Section \ref{sec:limit_problem_gamma3}.

\section{\textit{A priori} estimates}
\label{sec:apriori}

Existence of a unique weak solution $(\veps,\peps,\ueps)$ for the microscopic problem in the sense of Definition \ref{def:weak_solution_micro_model} is quite standard and can be obtained for example using the Galerkin-method. Also uniqueness is quite clear. Hence, in this section we focus on the derivation of $\vareps$-uniform \textit{a priori} estimates, which build the basis for the compactness results for the sequence of microscopic solutions. 
\begin{lemma}\label{lem:apriori_velocity_displacement}
The weak solution of $(\veps,\peps,\ueps)$ of the microscopic problem $\eqref{MicroscopicModel}$ fulfills for $\gamma \in \{1,3\}$ the following regularity and \textit{a priori} estimates
\begin{align*}
\|\partial_t \veps^{\pm}\|_{L^{\infty}((0,T),L^2(\oeps^{\pm}))} + \|\veps^{\pm}\|_{L^{\infty}((0,T),L^2(\oeps^{\pm}))} + \|\nabla \veps \|_{L^2((0,T) \times \oeps^{\pm})} &\le C,
\\
\|\partial_t \veps^M \|_{L^{\infty}((0,T),L^2(\oemf))} +   \frac{1}{\sqrt{\vareps}} \|\veps^M\|_{L^2((0,T)\times \oemf)} + \sqrt{\vareps} \|\nabla \veps^M\|_{L^2((0,T)\times \oemf)} &\le C,
\\
\|\partial_{tt} \ueps\|_{L^{\infty}((0,T),L^2(\oems))} + \frac{1}{\sqrt{\vareps}} \|\ueps\|_{H^1((0,T),L^2(\oems))} 
\\
+ \sqrt{\vareps} \|\nabla \ueps \|_{W^{1,\infty}((0,T),L^2(\oems))}  + \frac{1}{\vareps^{\frac{\gamma}{2}}} \|D(\ueps)\|_{W^{1,\infty}((0,T),L^2(\oems)) } &\le C.
\end{align*}
\end{lemma}
\begin{proof}
\textit{Step 1: (Basic estimates including the symmetric gradient)} We test the weak equation $\eqref{eq:Var_Micro_veps}$ with $\phieps= \veps $ in $\oef$ and $\phieps = \partial_t \ueps$ in $\oems$ to obtain for almost everywhere in $(0,T)$
\begin{align*}
\frac12 \frac{d}{dt}& \|\veps\|_{L^2(\oef)}^2 + \frac12 \frac{d}{dt} \|\partial_t \ueps\|_{L^2(\oems)}^2 + \sum_{\pm} \|D(\veps^{\pm})\|_{L^2(\oeps^{\pm})}^2 
\\
&+ \vareps\|D(\veps^M)\|_{L^2(\oemf)}^2 + \frac{1}{\vareps^{\gamma}} \int_{\oems} A_{\vareps} D(\ueps) : D(\partial_t \ueps ) dx = \sum_{\pm} \int_{\oeps^{\pm}} f_{\vareps}^{\pm} \cdot \veps^{\pm} dx.
\end{align*}
Using the coercivity of $A_{\vareps}$ from assumption \ref{AssumptionElasticityTensor}, we obtain after integration with respect to time for all $t \in (0,T)$
\begin{align*}
\|\veps(t)&\|^2_{L^2(\oef)} + \|\partial_t \ueps(t)\|_{L^2(\oems)}^2 + \sum_{\pm} \|D(\veps^{\pm})\|_{L^2((0,t)\times \oeps^{\pm})}^2 + \vareps \|D(\veps^M)\|_{L^2((0,t)\times \oemf)}^2 
\\
&+ \frac{1}{\vareps^{\gamma}} \|D(\ueps)(t)\|_{L^2(\oems)}^2 \le C \left( \sum_{\pm} \left\{\|f_{\vareps}^{\pm}\|^2_{L^2((0,T)\times \oeps^{\pm})} + \|\veps^{\pm}\|_{L^2((0,t)\times \oeps^{\pm})}^2 \right\} + \|\veps^0 \|_{L^2(\oef)}^2 \right).
\end{align*}
Now, Gronwall inequality and the assumptions on $f_{\vareps}^{\pm}$ and $\veps^0$ imply 
\begin{align*}
\|\veps &\|_{L^{\infty}((0,T),L^2(\oef))} + \|\partial_t \ueps \|_{L^{\infty}((0,T), L^2(\oems))} + \sum_{\pm} \|D(\veps^{\pm})\|_{L^2((0,T)\times \oeps^{\pm})} 
\\
&+ \sqrt{\vareps} \|D(\veps^M)\|_{L^2((0,T)\times \oemf)} + \frac{1}{\vareps^{\frac{\gamma}{2}}} \|D(\ueps)\|_{L^{\infty}((0,T),L^2(\oems)}^2 
\\
& \hspace{5em}\le C \left( \sum_{\pm} \|f_{\vareps}^{\pm}\|_{L^2((0,T)\times \oeps^{\pm})} + \|\veps^0\|_{L^2(\oef)}\right) \le C.
\end{align*}
Next, differentiating the microscopic system $\eqref{MicroscopicModel}$ with respect to time (respectively $\eqref{eq:Var_Micro_veps} $ with respect to time) and testing with $\phieps = \partial_t \veps $ in $\oef$ and $\phieps = \partial_{tt} \ueps$ in $\oems$, we get with similar arguments as above
\begin{align*}
\|\partial_t \veps &\|_{L^{\infty}((0,T),L^2(\oef))} + \|\partial_{tt} \ueps \|_{L^{\infty}((0,T), L^2(\oems))} + \sum_{\pm} \|D(\partial_t\veps^{\pm})\|_{L^2((0,T)\times \oeps^{\pm})} 
\\
&+ \sqrt{\vareps} \|D(\partial_t\veps^M)\|_{L^2((0,T)\times \oemf)} + \frac{1}{\vareps^{\frac{\gamma}{2}}} \|D(\partial_t\ueps)\|_{L^{\infty}((0,T),L^2(\oems)}^2 
\\
& \hspace{5em}\le C \left( \sum_{\pm} \|\partial_t f_{\vareps}^{\pm}\|_{L^2((0,T)\times \oeps^{\pm})} + \|\veps^0\|_{H^1(\oef)}\right) \le C.
\end{align*}

\textit{Step 2: (Improved estimates for $\veps^M$ and $\ueps$)} Next, we derive estimates for $\veps^M$ and better estimates for $\ueps$ resp. $\partial_t \ueps$. A direct use of the Korn inequality (see Lemma \ref{KornInequalityPerforatedLayer}) does not give the desired result. We first control $\veps^M$ via the bulk velocities and then in a second step $\ueps$ via the continuity of the velocities across $\geps$. 
We start with the estimate for $\nabla \veps^M$. For this, we define $\weps^M:= \veps^M - \partial_t \tueps$ with the extension $\tueps:= E_{\vareps} \ueps$, where $E_{\vareps}$ denotes the extension operator from Lemma \ref{lem:Extension_operators}. We emphasize that $E_{\vareps}$ is applied pointwise in time and the $L^2$-regularity of $\ueps$ implies that this operator commutes with the time-derivative, \ie we have $\partial_t E_{\vareps} \ueps = E_{\vareps} \partial_t \ueps$. We have $\weps = 0$ on $\geps$ and now the Korn inequalities in Lemma \ref{lem:Korn_inequality} and \ref{KornInequalityPerforatedLayer},  and the estimate for the extension operator from Lemma \ref{lem:Extension_operators} imply 
\begin{align*}
\|\nabla \veps^M\|_{L^2((0,T)\times \oemf)} &\le \|\nabla \weps^M \|_{L^2((0,T)\times \oemf)} + \|\nabla \partial_t \tueps\|_{L^2((0,T)\times \oemf)}
\\
&\le C\left( \|D(\weps^M)\|_{L^2((0,T)\times \oemf)} + \|\partial_t \nabla \ueps\|_{L^2((0,T)\times \oems)}\right)
\\
&\le C \left( \|D(\veps^M)\|_{L^2((0,T)\times \oemf)} + \|\partial_t D(\tueps)\|_{L^2((0,T)\times \oemf)} \right. \\
& \left. + \foe \|\partial_t D(\ueps)\|_{L^2((0,T)\times \oems)} \right) \le \frac{C}{\sqrt{\vareps}}.
\end{align*}
Now, from Lemma \ref{lem:estimate_L2_membrane_Bulk_Gradient} and the inequalities for $\veps^{\pm}$ obtained above we get
\begin{align*}
\frac{1}{\sqrt{\vareps}} \|\veps^M\|_{L^2((0,T)\times \oemf)} \le C \left( \sqrt{\vareps} \|\nabla \veps^M\|_{L^2((0,T)\times \oemf)} + \sum_{\pm} \|\veps^{\pm}\|_{H^1(\oeps^{\pm})}\right) \le C.
\end{align*}
Next, we estimate $\partial_t \ueps $ and $\ueps$. From  $\eqref{ineq:aux_trace_inequality}$ in Lemma \ref{lem:estimate_L2_membrane_Bulk_Gradient} and the trace inequality from Lemma \ref{lem:Stand_scaled_trace_inequality} we obtain with $\veps^M = \partial_t \ueps $ on $\geps$ (also we use the Korn inequality from Lemma \ref{KornInequalityPerforatedLayer})
\begin{align*}
\|\partial_t \ueps \|_{L^2((0,T) \times \oems)} &\le C \left( \vareps \|\partial_t \nabla \ueps\|_{L^2((0,T)\times \oems)} + \sqrt{\vareps} \|\veps^M\|_{L^2((0,T)\times \geps)}\right)
\\
&\le C  \left( \|D(\partial_t \ueps)\|_{L^2((0,T)\times \oems)} + \|\veps^M\|_{L^2((0,T)\times \oemf)} \right.\\
& \left. + \vareps\|\nabla \veps^M\|_{L^2((0,T)\times \oemf)} \right) \le C (\vareps^{\frac{\gamma}{2}} + \sqrt{\vareps}) \le C \sqrt{\vareps}.
\end{align*}
Using $\ueps(t) = \int_0^t \partial_t \ueps dt$ we obtain the  desired result.
\end{proof}

From the Korn inequality  (see Lemma \ref{KornInequalityPerforatedLayer}) we immediately obtain for $\gamma = 3$ the following inequality which improves the estimate for the horizontal components $\ueps^1$ and $\ueps^2$ compared to the case $\gamma = 1$.
\begin{corollary}
For $\gamma  =3$ the displacement $\ueps$ from the microscopic solution $(\veps,\peps,\ueps)$ of $\eqref{MicroscopicModel}$ fulfills
\begin{align*}
\sum_{i=1}^2 \frac{1}{\vareps^{\frac32}} \|\ueps^i\|_{W^{1,\infty}((0,T),L^2(\oems)}  \le C.
\end{align*}
\end{corollary}
\begin{proof}
   This is a direct consequence of the Korn inequality in Lemma \ref{KornInequalityPerforatedLayer} in the appendix.
\end{proof}

In the next Lemma we give an estimate for the fluid pressures $\peps^{\pm}$ and $\peps^M$. For Stokes-flow in thin perforated layers, see for example \cite{bayada1989homogenization} or \cite{fabricius2023homogenization}, the membrane pressure $\peps^M$ is usually of order $\sqrt{\vareps}$. Here, the coupling to the bulk regions implies that $\peps^M$ is of order $\vareps$ and therefore its two-scale limit vanishes (see also \cite{gahn2022derivation} and \cite{gahn2025effective}). 
\begin{lemma}\label{lem:apriori_pressure}
Let $\gamma \in \{1,3\}$ and let $(\veps,\peps,\ueps)$ be the microscopic solution of $\eqref{MicroscopicModel}$. Then, the fluid pressure $\peps$ fulfills
\begin{align*}
\sum_{\pm} \|\peps^{\pm}\|_{L^2((0,T)\times \oeps^{\pm})} + \foe \|\peps^M\|_{L^2((0,T)\times \oemf)} \le C.
\end{align*}
\end{lemma}
\begin{proof}
This is a direct consequence of Lemma \ref{lem:Bogovskii} and the \textit{a priori} estimates from Lemma \ref{lem:apriori_velocity_displacement}.
\end{proof}

\section{Compactness results}
\label{sec:compactness_results}

In this section we give the (two-scale) compactness results for the microscopic solutions $(\veps,\peps,\ueps)$. Here, for the sequence $\ueps$ we have to distinguish between the cases $\gamma = 1$ and $\gamma =3$. In the latter, due to the different scalings for the horizontal and the vertical displacements, as well as the scaling of order $\vareps^{\frac32}$ for the symmetric gradient, we obtain in the limit a Kirchhoff-Love displacement, see also \cite{gahn2022derivation}. In the case $\gamma = 1$ all components of the displacement and the symmetric gradient are of the same order. For this, we need the following two-scale compactness result (in which we ignore for a simpler notation the time-variable): 
\begin{lemma}\label{lem:two_scale_general_displ_membrane}
Let $\weps\in H^1(\oems,\partial_D \oems )^3$ with 
\begin{align*}
\frac{1}{\sqrt{\vareps}} \|\weps\|_{L^2(\oems)} + \frac{1}{\sqrt{\vareps}} \|D(\weps)\|_{L^2(\oems)} \le C.
\end{align*}
Then, there exists $w_0 \in H^1(\Sigma)^2 \times L^2(\Sigma)$ and $w_1 \in L^2(\Sigma,H_{\#}^1(Z_s)/\R)^3$ such that up to a subsequence
\begin{align*}
\chi_{\oems} w_{\vareps} \rats \chi_{Z_s} w_0,\qquad  \chi_{\oems} D(w_{\vareps}) \rats \chi_{Z_s} \left( D_{\x} (w_0) + D_y(w_1) \right).
\end{align*}
\end{lemma}
\begin{proof}
From the Korn inequality in Lemma \ref{KornInequalityPerforatedLayer} we get
\begin{align*}
    \frac{1}{\sqrt{\vareps}} \|\weps\|_{L^2(\oems)} + \sqrt{\vareps} \|\nabla \weps \|_{L^2(\oems)} \le C.
\end{align*}
Hence, there exists $w_0 \in L^2(\Sigma, H_{\#}^1(Z_s)/\R)^3$ such that up to a subsequence (see Lemma \ref{LemmaBasicTSCompactness} in the appendix) 
\begin{align*}
  \chi_{\oems} \weps \rats \chi_{Z_s} w_0,  \qquad \chi_{\oems}\vareps \nabla \weps \rats \chi_{Z_s} \nabla_y w_0. 
\end{align*}
The estimate for $D(\weps)$ immediately implies $D_y(w_0)=0$ and therefore $w_0$ is a rigid displacement. The  periodic boundary condition (with respect to the first two components) implies $w_0(\x,y) = w_0(\x)\in L^2(\Sigma)^3$. Now,  we define $\ueps:= \vareps \weps $, which fulfills
\begin{align*}
    \frac{1}{\vareps^{\frac32}} \|D(\ueps)\|_{L^2(\oeps^M)} \le C.
\end{align*}
An application of Lemma \ref{lem:two_scale_Kirchhoff_Love} from the appendix (which is also valid in the time-independent case) implies the existence of $u_0^3 \in H^2_0(\Sigma)$, $\hat{u}_1 = (u_1^1,u_1^2)\in H_0^1(\Sigma)^2 $, and $u_2 \in L^2(\Sigma,H_{\#}^1(Z_s)/\R)^3$ such that up to a subsequence (for $\alpha=1,2$)
\begin{align*}
    \chi_{\oems}\ueps^3 &\rats \chi_{Z_s}u_0^3,
\\
\chi_{\oems}\frac{\ueps^{\alpha}}{\epsilon} &\rats \chi_{Z_s}\big(u_1^{\alpha} - y_3 \partial_{\alpha} u_0^3\big),
\\
\frac{1}{\epsilon} \chi_{\oems} D(\ueps) &\rats  \chi_{Z_s} \left(D_{\x}(\hat{u}_1) - y_3 \nabla_{\x}^2 u_0^3 + D_y(u_2) \right).
\end{align*}
Comparing these convergences with the two-scale convergences of $\weps$ above and using $\ueps = \vareps \weps$, we immediately obtain $u_0^3 = 0$ and $w_0^{\alpha} = u_1^{\alpha} \in H^1_0(\Sigma)$ for $\alpha = 1,2$. With $w_1:= u_2$ we obtain the desired result.
\end{proof}

Now, as a direct consequence we obtain for $\gamma =1$ the following compactness result for the displacement and its time-derivative:
\begin{proposition}\label{prop:compactness_disp_gamma1}[Compactness displacement $\gamma = 1$]
Let $\gamma = 1$ and let $(\veps,\ueps,\peps)$ be the microscopic solution of $\eqref{MicroscopicModel}$. Then, there exist
\begin{align*}
    u_0 &\in H^1((0,T),H^1_0(\Sigma)^2 \times L^2(\Sigma))
    \\
    u_1 &\in H^1((0,T),L^2(\Sigma,H_{\#}^1(Z_s)/\R))^3,
\end{align*}
such that up to a subsequence it holds that (for $\alpha = 0,1$)
\begin{align*}
    \chi_{\oems} \partial_t^{\alpha}\ueps &\rats \chi_{Z_s} \partial_t^{\alpha} u_0, 
    \\
    \chi_{\oems} D( \partial_t^{\alpha}\ueps) &\rats \chi_{Z_s}\left( D_{\x} (\partial_t^{\alpha} \hat{u}_0) + D_y( \partial_t^{\alpha} u_1)\right)
\end{align*}
with $\hat{u}_0:= (u_0^1,u_0^2)$.
\end{proposition}
\begin{proof}
This is a direct consequence of Lemma \ref{lem:two_scale_general_displ_membrane} (which can be modified directly for the time-dependent case) and the \textit{a priori} estimates for the displacement in Lemma \ref{lem:apriori_velocity_displacement}.
\end{proof}
For the case $\gamma = 3$ we can directly use the bound for the symmetric gradient $D(\ueps)$ and obtain from Lemma \ref{lem:two_scale_Kirchhoff_Love}:
\begin{proposition}\label{prop:compactness_disp_gamma3}[Compactness displacement $\gamma = 3$]
Let $\gamma = 3$ and let $(\veps,\ueps,\peps)$ be the microscopic solution of $\eqref{MicroscopicModel}$. Then, there exist 
\begin{align*}
    u_0^3 &\in H^1((0,T),H^2_0(\Sigma)),
    \\
    \hat{u}_1 &\in H^1((0,T),H^1_0(\Sigma))^2,
    \\
    u_2 &\in H^1((0,T),L^2(\Sigma,H_{\#}^1(Z_s)/\R))^3,
\end{align*}
such that up to a subsequence it holds that 
\begin{align*}
    \chi_{\oems}\ueps^3 &\rats \chi_{Z_s} u_0^3,
\\
\chi_{\oems}\frac{\ueps^{\alpha}}{\epsilon} &\rats \chi_{Z_s}\big(u_1^{\alpha} - y_3 \partial_{\alpha} u_0^3\big),
\\
\frac{1}{\epsilon} \chi_{\oems} D(\ueps) &\rats  \chi_{Z_s} \left(D_{\x}(\hat{u}_1) - y_3 \nabla_{\x}^2 u_0^3 + D_y(u_2) \right),
\end{align*}
where $\hat{u}_1 = (u_1^1, u_1^2)$ and $\alpha=1,2$.
The same result is valid if  we consider $\partial_t \ueps$ and the time-derivatives of the limit functions.
\end{proposition}
\begin{proof}
Follows from Lemma \ref{lem:two_scale_Kirchhoff_Love} in the appendix and the \textit{a priori} estimates in Lemma \ref{lem:apriori_velocity_displacement}.
\end{proof}

Next, we formulate the convergence results for the bulk velocity and pressure.
\begin{proposition}\label{prop:compactness_bulk}[Compactness bulk velocity and pressure]
Let $\gamma \in \{1,3\}$ and let $(\veps,\ueps,\peps)$ be the microscopic solution of $\eqref{MicroscopicModel}$. Then, there exist
\begin{align*}
    v_0^{\pm} &\in L^2((0,T),H^1(\Omega^{\pm})^3 \cap H^1((0,T),L^2(\Omega^{\pm}))^3,
    \\
    p_0^{\pm} &\in L^2((0,T)\times \Omega^{\pm}),
\end{align*}
such that up to a subsequence
\begin{align*}
    \chi_{\oeps^{\pm}} \veps^{\pm} &\rightarrow v_0^{\pm} &\mbox{ in }& L^2((0,T)\times \Omega^{\pm})^3,
    \\
    \chi_{\oeps^{\pm}} \nabla \veps^{\pm} &\rightharpoonup \nabla v_0^{\pm} &\mbox{ weakly in }& L^2((0,T) \times \oeps^{\pm})^{3\times 3},
    \\
    \veps^{\pm}|_{S_{\vareps}^{\pm}} &\rightarrow v_0^{\pm}|_{\Sigma} &\mbox{ in }& L^2((0,T)\times \Sigma)^3,
    \\
    \chi_{\oeps^{\pm}} \partial_t \veps^{\pm} &\rightharpoonup \partial_t v_0^{\pm} &\mbox{ weakly in }& L^2((0,T)\times \Omega^{\pm})^3,
    \\
    \chi_{\oeps^{\pm}}\peps^{\pm} &\rightharpoonup p_0^{\pm} &\mbox{ weakly in }& L^2((0,T)\times \Omega^{\pm}).
\end{align*}
\end{proposition}
\begin{proof}
A proof for a similar result can be found in \cite{NeussJaeger_EffectiveTransmission}. However, for the sake of completeness we shortly sketch another proof of this result. Due to standard extension theorems for Sobolev functions we find an extension $\tveps^{\pm} \in L^2((0,T),H^1(\Omega^{\pm}))^3 \cap H^1((0,T),L^2(\Omega^{\pm}))^3$, which fulfills the same \textit{a priori} estimates as $\veps^{\pm}$ in Lemma \ref{lem:apriori_velocity_displacement}. Hence, using the Aubin-Lions Lemma, we can find limit functions $v_0^{\pm} \in L^2((0,T),H^1(\Omega^{\pm}))^3 \cap H^1((0,T),L^2(\Omega^{\pm}))^3$, such that up to a subsequence
\begin{align*}
 \tveps^{\pm} &\rightarrow v_0^{\pm} &\mbox{ in }& L^2((0,T)\times \Omega^{\pm})^3,
    \\
    \nabla \tveps^{\pm} &\rightharpoonup \nabla v_0^{\pm} &\mbox{ weakly in }& L^2((0,T) \times \oeps^{\pm})^{3\times 3},
    \\
    \tveps^{\pm}|_{\Sigma}^{\pm} &\rightarrow v_0^{\pm}|_{\Sigma} &\mbox{ in }& L^2((0,T)\times \Sigma)^3,
    \\
     \partial_t \tveps^{\pm} &\rightharpoonup \partial_t v_0^{\pm} &\mbox{ weakly in }& L^2((0,T)\times \Omega^{\pm})^3.
\end{align*}
This immediately implies all the desired convergences of $\veps^{\pm}$ except the convergence of the trace. With the mean value theorem we obtain
\begin{align*}
   \|\veps^{\pm}|_{S_{\vareps}^{\pm}} - \tveps^{\pm}|_{\Sigma} \|_{L^2((0,T)\times \Sigma)}  \le C \sqrt{\vareps} \|\partial_n \tveps^{\pm}\|_{L^2((0,T)\times \Omega^{\pm})} \le C\sqrt{\vareps}, 
\end{align*}
which implies the strong convergence of the trace $\veps^{\pm}|_{S_{\vareps}^{\pm}}$ to $v_0^{\pm}|_{\Sigma}$. The convergence of the pressure is obvious.
\end{proof}
Now, we formulate the compactness results for $\veps^M$ and $\peps^M$ together with  coupling conditions between the limit $v_0^M$ and $v_0^{\pm}$ respectively $\partial_t u_0$.

\begin{proposition}\label{prop:compactness_fluid_layer}
    Let $\gamma \in \{1,3\}$ and let $(\veps,\ueps,\peps)$ be the microscopic solution of $\eqref{MicroscopicModel}$. Then, there exists $v_0^M \in L^2((0,T)\times \Sigma, H_{\#}^1(Z_f))^3$ with $\nabla_y \cdot v_0^M = 0$ such that
    \begin{align*}
    \chi_{\oeps^{M,f}}\veps^M \rats \chi_{Z_f} v_0^M, \qquad \vareps \chi_{\oeps^{M,f}}\nabla \veps^M \rats \chi_{Z_f} \nabla_y v_0^M, \qquad
    \chi_{\oeps^{M,f}}\peps^M \rats 0.
    \end{align*}
Further, it holds that
\begin{align*}
    v_0^M &= v_0^{\pm} &\mbox{ on }& (0,T)\times \Sigma \times S^{\pm},
    \\
    v_0^M &= \partial_t u_0 &\mbox{ on }& (0,T)\times \Sigma \times \Gamma,
\end{align*}
with $u_0$ from Proposition \ref{prop:compactness_disp_gamma1} for $\gamma =1$ and $u_0 = (0,0,u_0^3)$ with $u_0^3$ from Proposition \ref{prop:compactness_disp_gamma3} for $\gamma =3$. In particular $v_0^M$ is constant on $S^{\pm}$ with respect to $y$.
\end{proposition}
\begin{proof}
The compactness results for $\veps^M$ and $\peps^M$ are direct consequences of the \textit{a priori} estimates from Lemma \ref{lem:apriori_velocity_displacement} and \ref{lem:apriori_pressure} together with Lemma \ref{LemmaBasicTSCompactness} in the appendix. The conditions on $S^{\pm}$ and $\Gamma$ follow from $\veps^M = \veps^{\pm}$ on $S_{\vareps}^{\pm}$ and $\veps^M = \partial_t \ueps$ on $\geps$, and the convergence results for $\veps^{\pm}$ and $\partial_t \ueps$. We refer to \cite{gahn2022derivation} for more details.
\end{proof}

\begin{remark}
We emphasize that here we have no information about $\nabla_{\x} \cdot \int_{Z_f} v_0^M dy$, which usually gives a Darcy law.
\end{remark}
As a consequence of the divergence condition on $v_0^M$ we obtain the continuity of the normal component of the velocity $v_0^{\pm}$ across $\Sigma$:
\begin{corollary}
The limit function $v_0^{\pm}$ from Proposition \ref{prop:compactness_bulk} fulfills
\begin{align*}
    [v_0^+]_3 = [v_0^-]_3\qquad \mbox{on } (0,T)\times \Sigma.
\end{align*}
\end{corollary}
\begin{proof}
It holds with the boundary conditions from Proposition \ref{prop:compactness_fluid_layer} that
\begin{align*}
[v_0^+]_3 - [v_0^-]_3 &= \sum_{\pm} \int_{S^{\pm}} v_0^M \cdot \nu_f d\sigma = \int_{\partial Z_f} v_0^M \cdot \nu_f d\sigma - \int_{\Gamma} \partial_t u_0 \cdot \nu_f d\sigma
\\
&= \int_{Z_f} \nabla_y \cdot v_0^M dy + \partial_t u_0 \cdot \underbrace{\int_{\Gamma} \nu_s d\sigma }_{= 0} = 0.
\end{align*}
\end{proof}

Based on the compactness results in this section, we are able to pass to the limit $\vareps \to 0$ in the weak equation $\eqref{eq:Var_Micro_veps}$. For this, we have to distinguish between the two cases $\gamma = 1$ and $\gamma = 3$.

\section{Derivation of the limit problem for $\gamma = 1$}
\label{sec:limit_problem_gamma1}

Let us consider the case $\gamma = 1$. Here, the most crucial point is the derivation of the macroscopic problem for the fluid flow. For this we will choose suitable test-functions $\phi$ in the microscopic equation $\eqref{eq:Var_Micro_veps}$, having a similar structure as the limit function $ (v_0^+,v_0^M,v_0^-)$. More precisely, we choose test-functions $\phi = (\phi^+,\phi^M,\phi^-)$ such that $\phi^+$ and $\phi^-$  are continuous in the third component across $\Sigma$, and $\phi^{\pm} = \phi^M$ on $\Sigma \times S^{\pm}$. To avoid strong oscillations in the stress term in the solid domain, we request that $\phi^M$ is constant with respect to the microscopic variable in the solid part $Z_s$ (we need that the symmetric gradient $D_y(\phi^M)$ with respect to $y$ vanishes in the solid part). To identify the effective coefficients in the macro model, we then use a specific decomposition for $\phi^M$, see $\eqref{eq:decomposition_phiM}$.
\\

We start with introducing the solution space for $(v_0^+,v_0^M,v_0^-)$, which is also the suitable space of test-functions in the limit problem: We define
\begin{align*}
\spaceV:= \bigg\{ \phi = &(\phi^+,\phi^M,\phi^-) \in H^1(\Omega^+,\partial_D \Omega^+)^3 \times L^2(\Sigma,H_{\#}^1(Z))^3 \times H^1(\Omega^-,\partial_D \Omega^-)^3,  
\\
&\phi^{\pm} = \phi^M \mbox{ on } \Sigma \times S^{\pm}, \, \phi^+_3 = \phi^-_3 \mbox{ on } \Sigma, \, \phi^M|_{Z_s} \in H^1_0(\Sigma)^3 \bigg\},
\end{align*}
together with the norm 
\begin{align*}
\| \phi\|_{\spaceV}:= \sum_{\pm} \|\phi^{\pm}\|_{H^1(\Omega^{\pm})} + \|\phi^M\|_{L^2(\Sigma,H^1(Z_f))} + \| \phi^M|_{Z_s} \|_{H^1(\Sigma)}.
\end{align*}
The condition for $\phi^M|_{Z_s} $ means that the function $\phi^M$ is constant with respect to the variable $y$ in $Z_s$.
We will see that this space is the appropriate function space for the limit problem. For $\phi \in \spaceV$ the function $\phi^M\left(\x,\fxe\right)$ is not an admissible test-function in the weak equation in the thin layer, due to missing measurability and the incorrect boundary condition on the lateral boundary $\partial_D \oemf$. However, in the limit problem for the fluid in the thin layer, boundary conditions on $\partial \Sigma$ are not necessary. To obtain suitable test-functions we have to choose smooth test-functions in $\spaceV$. We have the following density result:
\begin{lemma}\label{lem:density_spaceV}
The space 
\begin{align*}
 \spaceV^{\infty} :=     \spaceV \cap \left( C_0^{\infty}(\overline{\Omega^+}\setminus \partial_D \Omega^+)^3 \times C_0^{\infty}(\Sigma, C_{\#}^{\infty}(\overline{Z}))^3 \times  C_0^{\infty}(\overline{\Omega^-}\setminus \partial_D \Omega^-)^3 \right)
\end{align*}
is dense in $\spaceV$. 
Further, we have that the space $\left\{\phi \in \spaceV^{\infty} \, : \, \nabla_y \cdot \phi^M = 0\right\}$ is dense in the space $\left\{\phi \in \spaceV \, : \, \nabla_y \cdot \phi^M = 0 \right\}$. 
\end{lemma}
\begin{proof}
Let $\eta_i^{\pm}, \eta_i^{\Gamma} \in C_{\#}^{\infty}(\overline{Z_f})^3$ for $i=1,2,3$ with 
\begin{align*}
\eta_i^{\pm} &= e_i \quad\mbox{on } S^{\pm}, \qquad \eta_i^{\pm} = 0 \quad\mbox{ in a neighborhood of } \Gamma \cup S^{\mp},
\\
\eta_i^{\Gamma} &= e_i \quad\mbox{ on } \Gamma, \qquad \eta_i^{\Gamma} = 0 \quad\mbox{ in a neighborhood of } S^+ \cup S^-,
\end{align*}
constantly extended into $Z_s$.
Now, let $\phi \in \spaceV$. There exists a sequence $\phi^{\pm}_k \in C_0^{\infty}(\overline{\Omega^{\pm}} \setminus \partial_D \Omega^{\pm})^3$, such that $[\phi_k^+]_3 = [\phi_k^-]_3$ on $\Sigma$ and $\phi_k^{\pm} \rightarrow \phi^{\pm}$ in $H^1(\Omega^{\pm})^3$. Further, let $\mu^M:= \phi^M|_{Z_s} \in H_0^1(\Sigma)^3$. Hence, there exists a sequence $\mu^M_k \in C_0^{\infty}(\Sigma)^3$ such that $\mu_k^M \rightarrow \mu^M$ in $H^1(\Sigma)^3$. Now, we define (note that for a better readability, in the formula below we indicate by $[]_i$ the $i$-th component of a function)
\begin{align*}
\eta^M (\x,y):= \sum_{i=1}^3 [\mu^M]_i(\x) \eta_i^{\Gamma} + \sum_{\alpha \in \{\pm\}} \sum_{i=1}^3 [\phi^{\pm}]_i(\x) \eta_i^{\pm}(y)
\end{align*}
for almost every $(\x,y) \in \Sigma \times Z$. In the same way we define $\eta^M_k \in C_0^{\infty}(\Sigma,C_{\#}^{\infty}(\overline{Z}))^3$ with the coefficients $\mu^M$ and $\phi^{\pm}$ replaced by $\mu_k^M$ and $\phi^{\pm}_k$. Obviously, we have $\eta^M_k \rightarrow \eta^M$ in $L^2(\Sigma,H^1(Z_f))^3$. Next, we define 
\begin{align*}
    w^M := \phi^M - \eta^M.
\end{align*}
Using the properties of $\phi^M$ and the definition of $\eta^M$, we easily obtain $w^M \in L^2(\Sigma, H_{\#}^1(Z))^3$ such that $w^M = 0$ on $\Gamma \cup S^+ \cup S^-$. Hence, there exists a sequence $w_k^M \in C_0^{\infty}(\Sigma,C_{\#}^{\infty}(Z_f))^3$ and zero in a neighborhood of $\Gamma \cup S^+ \cup S^-$ (and extended by zero to $Z_s$) such that $w_k^M \to w^M$ in $L^2(\Sigma, H^1(Z))^3$. Now, we define 
\begin{align*}
    \phi^M_k:= \eta_k^M + w_k^M
\end{align*}
and obtain 
\begin{align*}
    \|\phi^M - \phi^M_k\|_{L^2(\Sigma,H^1(Z_f))} \le \|w^M - w_k^M \|_{L^2(\Sigma,H^1(Z_f))} + \|\eta^M - \eta^M_k\|_{L^2(\Sigma,H^1(Z_f))}  \rightarrow 0
\end{align*}
for $k\to \infty$. Further, we have $\phi^M_k|_{Z_s} = \eta^M_k \rightarrow \eta^M = \phi^M|_{Z_s}$ in $H^1(\Sigma)^3$. Altogether, $(\phi^+_k,\phi^M_k,\phi_k^-)$ is an approximation of $\phi$ in $\spaceV$.

It remains to consider the case of divergence-free functions. The proof follows the same lines, where we only need a slight modification of the functions $\eta_i^{\pm}$ and $\eta_i^{\Gamma}$, which we have to choose divergence-free. However, for $\eta_3^{\pm}$ this is not possible and therefore we have to replace this functions by
\begin{align*}
    \eta_3:=  e_3 \mbox{ on } S^+ \cup S^-,\qquad \eta_3 = 0 \mbox{ in a neighborhood of } \Gamma.
\end{align*}
Then we can consider 
\begin{align*}
\eta^M (\x,y):= \sum_{i=1}^3 [\mu^M]_i(\x) \eta_i^{\Gamma} + \sum_{\alpha \in \{\pm\}} \sum_{i=1}^2 [\phi^{\pm}]_i(\x) \eta_i^{\pm}(y) + \phi_3 (\x) \eta_3(y)
\end{align*}
with $\phi_3:= \phi_3^+|_{\Sigma} = \phi_3^-|_\Sigma$ (here we use the continuity of the normal component of $\phi$), and in a similar way we define $\eta_k^M$. The rest of the proofs follows the same lines as above (now $w^M$  is divergence-free and can be approximated by smooth divergence-free functions).
\end{proof}
In the following, we derive a weak problem for the limit functions $(v_0^+, v_0^M, v_0^-)$, $(p_0^+, p_1^M, p_0^-)$, and $(u_0, u_1)$ from Proposition  \ref{prop:compactness_disp_gamma1}, \ref{prop:compactness_bulk} and\ref{prop:compactness_fluid_layer}, with test-functions from the space $\spaceV$. Here, $p_1^M$ is a first order corrector for the fluid pressure in the membrane and will be obtained below via the closed range theorem and the surjectivity of some specific divergence operator adapted to the coupling between the bulk domains and the interface. From the weak equation derived in this way, all macroscopic equations will be obtained by choosing specific test-functions in $\spaceV$. First of all, let us consider in the weak equation $\eqref{eq:Var_Micro_veps}$ the test-function
\begin{align*}
\phieps(x):= \begin{cases}
\phi^{\pm} \left( t,x \mp \vareps e_n\right) &\mbox{ for } x \in \oeps^{\pm},
\\
\phi^M\left(t,\x,\fxe\right) &\mbox{ for } x \in \oem,
\end{cases}
\end{align*}
with $\phi \in C_0^{\infty}((0,T) , \spaceV^{\infty})$ and $\nabla_y \cdot \phi^M = 0$. Obviously, this function is an admissible test-function.
We obtain (since $D_y(\phi^M) = 0$ in $Z_s$) almost everywhere in $(0,T)$
\begin{align*}
\sum_{\pm}& \bigg\{ \int_{\oeps^{\pm}} \partial_t \veps^{\pm}\cdot  \phi^{\pm}(x \mp \vareps e_n)  dx +  \int_{\oeps^{\pm}} D(\veps^{\pm}): D(\phi^{\pm})(x \mp \vareps e_n) dx   -\int_{\oeps^{\pm}} \peps^{\pm} \nabla \cdot \phi^{\pm}(x \mp \vareps e_n) dx \bigg\}
\\
+& \int_{\oemf} \partial_t \veps^M \cdot \phi^M  \bxfxe dx  + \int_{\oems} \partial_{tt} \ueps \cdot \phi^M\bxfxe dx
\\
+& \foe\int_{\oemf} \vareps D(\veps^M) : \left[ \vareps D_{\x}(\phi^M) + D_y(\phi^M)\right]\bxfxe  dx
-  \foe \int_{\oemf} \peps^M \nabla_{\x} \cdot \phi^M \bxfxe dx 
\\
+& \foe \int_{\oems} A_{\vareps} D(\ueps) : D_{\x}(\phi^M)\bxfxe dx
= \sum_{\pm} \int_{\oeps^{\pm}} f_{\vareps}^{\pm} \cdot \phi^{\pm}(x \mp \vareps e_n) dx .
\end{align*}
Before we pass to the limit $\vareps \to 0$ let us point out some crucial aspects regarding the choice of the test-functions.
\begin{remark}
The case of a rigid solid was treated in \cite{gahn2025effective}. There, only test-functions in $\spaceV$ were considered vanishing on $Z_s$. Here, we have to take into account also functions which are constant with respect to the microscopic variable $y$ in $Z_s$. This will lead to a membrane equation for the macroscopic displacement $u_0$ on $\Sigma$. 
\end{remark}

Using the compactness results in Section \ref{sec:compactness_results} and the assumptions \ref{ass:rhs_feps_pm} on $f_{\vareps}^{\pm}$, we obtain  after integration with respect to time for $\vareps \to 0$
\begin{align}
\begin{aligned}\label{eq:limit_general}
\sum_{\pm} & \left\{ \int_0^T \int_{\Omega^{\pm}} \partial_t v_0^{\pm} \cdot \phi^{\pm} dx dt + \int_0^T \int_{\Omega^{\pm}} D(v_0^{\pm}) : D(\phi^{\pm}) dxdt - \int_0^T  \int_{\Omega^{\pm}} p_0^{\pm} \nabla \cdot \phi^{\pm} dx  dt \right\} 
\\
+& \int_0^T \int_{\Sigma} \int_{Z_f} D_y(v_0^M) : D_y(\phi^M) dy d\x dt  
\\ &+ \int_0^T \int_{\Sigma} \int_{Z_s} A \left[D_{\x}(\hat{u}_0) + D_y(u_1)\right]: D_{\x} (\phi^M|_{Z_s}) dy d\x dt 
\\ &= \sum_{\pm} \int_0^T \int_{\Omega^{\pm}} f^{\pm}_0 \cdot \phi^{\pm} dx dt .
\end{aligned}
\end{align}
By density this result is valid for test functions $\phi = (\phi^+,\phi^M , \phi^-) \in L^2((0,T),\spaceV)$ with $\nabla_y \cdot \phi^M =0$. Note that, if we choose test-functions with $\phi^M = 0$ in $Z_s$, we obtain the same equation as in \cite[(17)]{gahn2025effective}. To construct an associated pressure, we argue in the same way as in \cite[Lemma 4]{gahn2025effective} (where the space $\spaceH$, consisting of functions vanishing in $Z_s$, is replaced by $\spaceV$), so we only give the main ideas. We consider the operator
\begin{align*}
\Div_{\spaceV}: \spaceV \rightarrow L^2(\Sigma \times Z_f), \qquad \Div_{\spaceV} \phi= \nabla_y \cdot \phi^M.
\end{align*}
Using that for $\phi \in \spaceV$ the function $\phi^M$ is constant on $\Gamma$ (with respect to $y$) and $\phi^M_3$ coincides on $S^+$ and $S^-$, we get
\begin{align*}
    \int_{Z_f} \Div_{\spaceV} \phi \,dy = 0.
\end{align*}
Hence, the classical Bogovskii-operator implies that $\Div_{\spaceV}$ is surjective onto $L^2(\Sigma,L_0^2(Z_f))$  (remember $L_0^2$ denotes the $L^2$-space with functions having mean value zero). This implies the existence of $p_1^M \in L^2((0,T)\times \Sigma,L_0^2(Z_f))$  such that for all $\phi \in \spaceV$ it holds almost everywhere in $(0,T)$
\begin{align}
\begin{aligned}\label{eq:two_scale_model_fluid}
\sum_{\pm} & \left\{ \int_{\Omega^{\pm}} \partial_t v_0^{\pm} \cdot \phi^{\pm} dx +\int_{\Omega^{\pm}} D(v_0^{\pm}) : D(\phi^{\pm}) dx -  \int_{\Omega^{\pm}} p_0^{\pm} \nabla \cdot \phi^{\pm} dx   \right\} 
\\
&+ \int_{\Sigma} \int_{Z_f} D_y(v_0^M) : D_y(\phi^M) dy d\x - \int_{\Sigma} \int_{Z_f} p_1^M \nabla_y \cdot \phi^M dy d\x 
\\
&+ \int_{\Sigma} \int_{Z_s} A \left[D_{\x}(\hat{u}_0) + D_y(u_1)\right]: D_{\x} (\phi^M|_{Z_s}) dy d\x = \sum_{\pm} \int_{\Omega^{\pm}} f^{\pm}_0 \cdot \phi^{\pm} dx  . 
\end{aligned}
\end{align}
This identity is now the basis for the derivation of all macroscopic equations by choosing suitable test-functions for $\phi^M$. We start with the derivation of an equation for the fluid flow by choosing $\phi^M = 0$ in $Z_s$ (in this case the integral over $Z_s$ vanishes), and we summarize the result in the following proposition:
\begin{proposition}\label{prop:two_scale_model_fluid}
The limit functions $(v_0^+,v_0^M ,v_0^-)$ and $(p_0^+,p_1^M,p_0^-)$ from Proposition \ref{prop:compactness_bulk} and \ref{prop:compactness_fluid_layer} fulfill $(v_0^+,v_0^M ,v_0^-) \in L^2((0,T),\spaceV)$, and represent the unique weak solution of 
\begin{align*}
\partial_t v_0^{\pm} - \nabla\cdot D(v_0^{\pm}) + \nabla p_0^{\pm} &= f_0^{\pm} &\mbox{ in }& (0,T)\times \Omega^{\pm},
\\
\nabla \cdot v_0^{\pm} &= 0 &\mbox{ in }& (0,T)\times \Omega^{\pm},
\\
[v_0^+]_3 &= [v_0^-]_3 &\mbox{ on }& (0,T)\times \Sigma,
\\
-[D(v_0^{\pm}) - p_0^{\pm}I] \cdot \nu &= 0 &\mbox{ on }& (0,T)\times \partial_N \Omega^\pm,
\\
-\nabla_y \cdot D_y(v_0^M) + \nabla_y p_1^M &= 0 &\mbox{ in }& (0,T)\times \Sigma \times Z_f,
\\
\nabla_y \cdot v_0^M &= 0 &\mbox{ in }& (0,T)\times \Sigma \times Z_f,
\\
v_0^M &= \begin{cases}
 v_0^{\pm}  &\mbox{ on } (0,T)\times \Sigma \times S^{\pm},
 \\
 \partial_t u_0  &\mbox{ on } (0,T)\times \Sigma \times \Gamma,
\end{cases} 
\\
-[D(v_0^{\pm}) - p_0^{\pm}I]\nu &= -[D_y(v_0^M) - p_1^M I] \nu &\mbox{ on }& (0,T)\times \Sigma \times S^{\pm},
\\
v_0^{\pm}(0) &= 0 &\mbox{ in }& \Omega^{\pm},
\\
 v_0^M \,\, Y\mbox{-periodic},
\end{align*}
with $u_0$ given in Proposition \ref{prop:compactness_disp_gamma1}.
A weak solution of this problem is a function $(v_0^+,v_0^M ,v_0^-)\in L^2((0,T),\spaceV)$ with $\partial_t v_0^{\pm} \in L^2((0,T)\times \Omega^{\pm})$ and $\nabla \cdot v_0^{\pm} = 0$ resp. $\nabla_y \cdot v_0^M = 0$, and satisfying the transmission condition on $\Gamma$, together with a pressure $(p_0^+,p_1^M,p_0^-) \in L^2((0, T)\times\Omega^+)\times L^2((0, T)\times\Sigma,L_0^2(Z_f)) \times L^2((0, T)\times\Omega^-)$, such that $\eqref{eq:two_scale_model_fluid}$ is valid for all $\phi \in \spaceV$ with $\phi^M = 0$ in $Z_s$, almost everywhere in $(0,T)$. 
\end{proposition}
A similar problem with $v_0^M = 0$ on $\Gamma$ was obtained in \cite[Proposition 9]{gahn2025effective} from the weak formulation \cite[(18))]{gahn2025effective}. A crucial difference to the weak equation \cite[(18))]{gahn2025effective}(apart from the different boundary condition on $\Gamma$)  is that the weak equation $\eqref{eq:two_scale_model_fluid}$ also allows the choice of test-functions $\phi^M$ which are constant on the solid part $Z_s$.

Next, we derive a representation for $v_0^M$ via $v_0^{\pm}$ and  $\partial_t u_0$, and suitable cell problems, which are introduced in the following: 
The tuple $(q_i^{\Gamma},\pi_i^{\Gamma}) \in H^1_{\#}(Z_f)^3 \times L_0^2(Z_f)$ for $i=1,2,3$ is the unique weak solution of the cell problem
\begin{align}
\begin{aligned}\label{eq:Cell_problem_q_i_Gamma}
-\nabla_y \cdot (D_y(q_i^{\Gamma})) + \nabla_y \pi_i^{\Gamma} &= 0 &\mbox{ in }& Z_f,
\\
\nabla_y \cdot q_i^{\Gamma} &= 0 &\mbox{ in }& Z_f,
\\
q_i^{\Gamma} &= \begin{cases}
    0 &\mbox{ on } S^+ \cup S^-,
    \\
    e_i &\mbox{ on } \Gamma,
\end{cases}
\\
(q_i^{\Gamma},\pi_i^{\Gamma}) \,\, Y\mbox{-periodic}.
\end{aligned}    
\end{align}
Further the tuple $(q_i^{\pm}, \pi_i^{\pm}) \in H^1_{\#}(Z_f)^3 \times L_0^2(Z_f)$ for $i=1,2$ is the unique weak solution of the cell problem
\begin{align}
\begin{aligned}\label{eq:Cell_problem_q_i_pm}
-\nabla_y \cdot (D_y(q_i^{\pm})) + \nabla_y \pi_i^{\pm} &= 0 &\mbox{ in }& Z_f,
\\
\nabla_y \cdot q_i^{\pm} &= 0 &\mbox{ in }& Z_f,
\\
q_i^{\pm} &= \begin{cases}
    0 &\mbox{ on } \Gamma,
    \\
    e_i &\mbox{ on } S^{\pm},
    \\
    0 &\mbox{ on } S^{\mp},
\end{cases}
\\
(q_i^{\pm},\pi_i^{\pm}) \,\, Y\mbox{-periodic}.
\end{aligned}    
\end{align}
Finally, the tuple $(q_3,\pi_3) \in H_{\#}^1(Z_f)^3 \times L_0^2(Z_f)$ is the unique weak solution of the cell problem
\begin{align}
\begin{aligned}\label{eq:Cell_problem_q_3}
-\nabla_y \cdot (D_y(q_3)) + \nabla_y \pi_3 &= 0 &\mbox{ in }& Z_f,
\\
\nabla_y \cdot q_3 &= 0 &\mbox{ in }& Z_f,
\\
q_3 &= \begin{cases}
    e_3 &\mbox{ on } S^+ \cup S^-,
    \\
    0 &\mbox{ on } \Gamma,
\end{cases}
\\
(q_3,\pi_3) \,\, Y\mbox{-periodic}.
\end{aligned}
\end{align}
Since 
\begin{align*}
    \int_{\partial Z_f} q_i^{\Gamma} \cdot \nu d\sigma =  \int_{\partial Z_f} q_i^{\pm} \cdot \nu d\sigma =  \int_{\partial Z_f} q_3 \cdot \nu d\sigma = 0,
\end{align*}
it is easy to check that all these problems admit a unique weak solution. To shorten the following notations, it will be suitable to introduce the notation 
\begin{align*}
    q_3^{\pm}:= \frac12 q_3, \qquad \pi_3^{\pm}:= \frac12 \pi^3.
\end{align*}

\begin{remark}\label{rem:linear_comb_q_i}
We point out that suitable linear combinations between $q_i^{\Gamma}$, $q_i^{\pm}$ and $q_3$ are constant. More precisely, we have for $i=1,2,3$
\begin{align}
     q_i^{\Gamma} + \sum_{\alpha \in \{\pm\}} q_i^{\alpha } = e_i.
\end{align}
In fact, we define $\tilde{q}_i$ for $i=1,2,3$ via
\begin{align*}
    \tilde{q}_i:= q_i^{\Gamma} + \sum_{\alpha\in \{\pm\}} q_i^{\alpha}.
\end{align*}
and in a similar way $\tilde{\pi}_i$ by $\pi_i^{\Gamma}$ and $\pi_i^{\pm}$. Then, we have for $i=1,2,3$ that $(\tilde{q}_i,\tilde{\pi}_i)$ fulfills the cell problem
\begin{align*}
    -\nabla_y \cdot D_y(\tilde{q}_i) + \nabla_y \tilde{\pi}_i &= 0 &\mbox{ in }& Z_f,
    \\
    \nabla_y \cdot \tilde{q}_i &= 0 &\mbox{ in }& Z_f,
    \\
    \tilde{q}_i &= e_i &\mbox{ on }& S^+ \cup S^- \cup \Gamma,
    \\
    (\tilde{q}_i,\tilde{\pi}_i) \,\, Y\mbox{-periodic}.
\end{align*}
Obviously, the unique weak solution of this problem is $(e_i,0)$ and therefore $\tilde{q}_i = e_i$.
In particular, for given $\theta \in H^1(\Omega,\partial_D \Omega)^3$ with $\theta|_{\Sigma} \in H^1_0(\Sigma)^3$ 
we have $\phi = (\phi^+,\phi^M ,\phi^-) = \theta \, \in \spaceV$ (this means $\phi^{\pm} = \theta^{\pm} $ in $\Omega^{\pm}$ and $\phi^M = \theta|_{\Sigma}$ in $\Sigma \times Z$) for 
\begin{align*}
    \phi^M = \sum_{i=1}^3 [\theta|_{\Sigma}]_i q_i^{\Gamma} + \sum_{\alpha \in \{\pm\}} \sum_{i=1}^3 [\theta|_{\Sigma}]_i q_i^{\alpha} .
\end{align*}
\end{remark}
\begin{remark}\label{representation-with_w}
For the derivation of the cell problem for $(v_0^M,p_1^M)$ as well as of the macroscopic problems for $(v_0^{\pm},p_0^{\pm})$ and $u_0$, 
we will choose different test-functions $\phi \in L^2((0,T),\spaceV)$ in $\eqref{eq:two_scale_model_fluid}$. Hereby, we use the fact that we can decompose every $\phi^M$ for $\phi \in \spaceV$ in the following way (for almost every $(\x,y) \in \Sigma \times Z$):
\begin{align}\label{eq:decomposition_phiM}
     \phi^M(\x,y) := \sum_{i=1}^3 \psi_i(\x) q_i^{\Gamma}(y) + \sum_{\alpha \in\{\pm\}} \sum_{i=1}^3 \phi_i^{\alpha} (\x) q_i^{\alpha} (y) + w^M
\end{align}
with $w^M \in H^1_{\#}(Z_f,\Gamma \cup S^+ \cup S-)^3$, $\psi \in H_0^1(\Sigma)^3$ and $\phi^{\pm} \in H^1(\Omega^{\pm},\partial_D \Omega^{\pm})^3$ with $\phi_3^+ = \phi_3^-$ on $\Sigma$. This can be easily seen by choosing $\psi:= \phi^M|_{Z_s} \in H_0^1(\Sigma)^3$ and noticing that, due to the boundary conditions of $q_i^{\Gamma}$ and $q_i^{\pm}$,
\begin{align*}
    w^M:= \phi^M -  \sum_{i=1}^3 \phi^M|_{Z_s} q_i^{\Gamma} + \sum_{\alpha \in\{\pm\}} \sum_{i=1}^3 \phi_i^{\alpha}  q_i^{\alpha}   \in H^1(Z_f,\Gamma \cup S^+ \cup S^-)^3.
\end{align*}
We emphasize that $\nabla_y \cdot w^M = \nabla_y \cdot \phi^M$.
\end{remark}
Now, we formulate the cell problem for $(v_0^M,p_1^M)$ and give a representation for both functions.
\begin{proposition}\label{prop:representation_v_0^M}
The limit function $v_0^M$ in Proposition \ref{prop:two_scale_model_fluid} fulfills 
\begin{align*}
-\nabla_y \cdot D_y(v_0^M) + \nabla_y p_1^M &= 0 &\mbox{ in }& (0,T)\times \Sigma \times Z_f,
\\
\nabla_y \cdot v_0^M &= 0 &\mbox{ in }& (0,T)\times \Sigma \times Z_f,
\\
v_0^M &= \begin{cases}
 v_0^{\pm}  &\mbox{ on } (0,T)\times \Sigma \times S^{\pm},
 \\
 \partial_t u_0  &\mbox{ on } (0,T)\times \Sigma \times \Gamma,
\end{cases} 
\\
 v_0^M \,\, Y\mbox{-periodic}. 
\end{align*}
Further, we have the following representations for almost every $(t,\x,y) \in (0,T)\times \Sigma \times Z_f$
\begin{align*}
v_0^M(t,\x,y) &= \sum_{i=1}^3 \partial_t u_0^i (t,\x) q_i^{\Gamma}(y) + \sum_{\alpha \in \{\pm\}} \sum_{i=1}^3  [v_0^{\alpha}]_i(t,\x,0) q_i^{\alpha}(y),
\\
p_1^M(t,\x,y) &= \sum_{i=1}^3 \partial_t u_0^i (t,\x) \pi_i^{\Gamma}(y) + \sum_{\alpha \in \{\pm\}} \sum_{i=1}^3  [v_0^{\alpha}]_i(t,\x,0) \pi_i^{\alpha}(y) .
\end{align*}
We emphasize that for the terms including $i=3$ for $\pm$, we can also write
\begin{align*}
    \sum_{\alpha\in\{\pm\}}  [v_0^{\alpha}]_3 q_3^{\alpha} = [v_0]_3 q_3
\end{align*}
with $[v_0]_3:= [v_0^+]_3|_\Sigma = [v_0^-]_3|_\Sigma$, and similar for the term including the pressure.
\end{proposition}
\begin{proof}
The equation for $(v_0^M,p_1^M)$ follows directly from $\eqref{eq:two_scale_model_fluid}$ by choosing $(\phi^+,\phi^M,\phi^-) = (0,\phi^M,0)$ with $\phi^M = w^M$ in the representation \eqref{eq:decomposition_phiM}. By uniqueness, we obtain the desired representation.
\end{proof}

\begin{remark}\label{rem:representation_v_0^M}
Proposition \ref{prop:representation_v_0^M} is also valid for $\gamma = 3$, if we put $u_0 = (0,0,u_0^3)$ with $u_0^3$ from Proposition \ref{prop:compactness_disp_gamma3}.
\end{remark}

\subsubsection*{Derivation of the macroscopic equation $\eqref{def:Macro_Stokes_model_strong}$ for $(v_0^{\pm},p_0^{\pm})$:}
Let $\phi^{\pm} \in H^1(\Omega^{\pm},\partial_D \Omega^{\pm})^3$ with $\phi^+_3 = \phi^-_3 $ on $\Sigma$. We test equation $\eqref{eq:two_scale_model_fluid}$ with $\phi = (\phi^+,\phi^M,\phi^-)$, where for $\phi^M$ we consider a decomposition $\eqref{eq:decomposition_phiM}$ of the form
\begin{align*}
\phi^M = \sum_{\alpha \in \{\pm\}} \sum_{i=1}^3 \phi_i^{\alpha} q_i^{\alpha} .
\end{align*}
In particular, $\phi^M$ vanishes in the solid part $Z_s$.
We get 
\begin{align}
\begin{aligned}\label{eq:aux_derivation_Stokes_macro}
\sum_{\pm} & \left\{ \int_{\Omega^{\pm}} \partial_t v_0^{\pm} \cdot \phi^{\pm} dx +\int_{\Omega^{\pm}} D(v_0^{\pm}) : D(\phi^{\pm}) dx -  \int_{\Omega^{\pm}} p_0^{\pm} \nabla \cdot \phi^{\pm} dx   \right\} 
\\
&+ \int_{\Sigma} \int_{Z_f} D_y(v_0^M) : D_y(\phi^M) dy d\x  = \sum_{\pm} \int_{\Omega^{\pm}} f^{\pm}_0 \cdot \phi^{\pm} dx  . 
\end{aligned}
\end{align}
Now, we use the representation of $\phi^M$ and $v_0^M$ from Proposition \ref{prop:representation_v_0^M} to rewrite the integral including $v_0^M$. Similar calculations can be found in \cite[Section 6.1]{gahn2025effective}, so we don't give all the details.

We obtain for almost every $(t,\x) \in (0,T)\times \Sigma$
\begin{align}
\begin{aligned}\label{eq:aux_derivation_macro_fluid_Bpm}
    \int_{Z_f}& D_y(v_0^M) : D_y(\phi^M) = \sum_{\alpha \in \{\pm\}} \sum_{i,j=1}^3 \partial_t u_0^i \phi_j^{\alpha} \int_{Z_f} D_y(q_i^{\Gamma}) : D_y(q_j^{\alpha} ) dy
\\
&+ \sum_{\alpha,\beta\in\{\pm\}} \sum_{i,j=1}^3 \phi^{\beta}_j [v_0^{\alpha}]_i \int_{Z_f} D_y(q_i^{\alpha}) : D_y(q_j^{\beta})dy 
\\
&= \sum_{\alpha \in \{\pm\}} L^{\alpha}\partial_t u_0 \cdot \phi^{\alpha} + \sum_{\alpha,\beta \in \{\pm\}} B^{\alpha,\beta}v_0^{\alpha} \cdot \phi^{\beta},
\end{aligned}
\end{align}
with the effective coefficients $L^{\alpha}, B^{\alpha,\beta} \in \R^{3\times 3}$ for $\alpha , \beta \in \{\pm\}$ defined for $i,j=1,2,3$ by
\begin{align}
    \label{Def:B}
    B^{\alpha,\beta}_{ji} &:= \int_{Z_f} D_y(q_i^{\alpha}) : D_y(q_j^{\beta})dy ,
    \\
    \label{Def:L}
    L_{ji}^{\alpha} &:= \int_{Z_f}D_y(q_i^{\Gamma}) : D_y(q_j^{\alpha})dy.
\end{align}
Altogether, we obtain the following equation for $(v_0^{\pm},p_0^{\pm})$ which holds almost everywhere in $(0,T)$:
\begin{align}
\begin{aligned}\label{eq:var_macro_fluid_with_B}
\sum_{\pm} & \left\{ \int_{\Omega^{\pm}} \partial_t v_0^{\pm} \cdot \phi^{\pm} dx +\int_{\Omega^{\pm}} D(v_0^{\pm}) : D(\phi^{\pm}) dx -  \int_{\Omega^{\pm}} p_0^{\pm} \nabla \cdot \phi^{\pm} dx   \right\} 
\\
&+ \int_{\Sigma} \sum_{\pm} L^{\pm} \partial_t u_0  \cdot \phi^{\pm} + \sum_{\alpha,\beta\in \{\pm\}} B^{\alpha, \beta} v_0^{\alpha} \cdot \phi^{\beta} \, d\x  = \sum_{\pm} \int_{\Omega^{\pm}} f^{\pm}_0 \cdot \phi^{\pm} dx  
\end{aligned}
\end{align}
for all $(\phi^+,\phi^-) \in H^1(\Omega^+,\partial_D \Omega^+)^3 \times H^1(\Omega^-,\partial_D \Omega^-)^3$ with $\phi^+_3 = \phi^-_3$ on $\Sigma$. Let us decompose the term on $\Sigma$ in its tangential and normal part. For a vector field $\psi$ we define its normal part by $\psi_{\nu}:= (\psi \cdot \nu) \nu$ and its tangential part by $\psi_t := \psi - \psi_{\nu}$. Hence, we obtain ($\phi^{\pm} \cdot \nu^{\pm} = \mp \phi_3$ for $\phi_3:= \phi^+_3|_\Sigma = \phi^-_3|_\Sigma$)
\begin{align*}
\int_{\Sigma} \sum_{\pm} L^{\pm}\partial_t & u_0 \cdot \phi^{\pm}+ \sum_{\alpha,\beta\in \{\pm\}} B^{\alpha, \beta} v_0^{\alpha} \cdot \phi^{\beta} d\x 
\\
=& \int_{\Sigma} \sum_{\beta \in \{\pm\}}\left[ L^{\beta} \partial_t u_0 + \sum_{\alpha \in\{\pm\}} B^{\alpha,\beta} v_0^{\alpha} \right] \cdot \nu^- \phi_3 d\x
\\
&+ \int_{\Sigma} \sum_{\beta \in \{\pm\}} \left[ L^{\beta} \partial_t u_0  + \sum_{\alpha\in \{\pm\}} B^{\alpha,\beta} v_0^{\alpha} \right] \cdot \phi_t^{\beta} d\x.
\end{align*}
Hence, $\eqref{eq:var_macro_fluid_with_B}$ is the weak formulation of problem $\eqref{def:Macro_Stokes_model_strong}$ with the boundary conditions  $\eqref{Macro_Model_Fluid_Stress_normal}$ and $\eqref{Macro_Model_Fluid_Stress_tangential}$
\begin{align*}
-\llbracket (D (v_0) - p_0 I ) \nu\cdot \nu\rrbracket &=
 \sum_{\beta \in \{\pm\}}\left[ L^{\beta} \partial_t u_0 + \sum_{\alpha \in\{\pm\}} B^{\alpha,\beta} v_0^{\alpha} \right] \cdot \nu^- 
&\mbox{ on }& (0,T)\times \Sigma,
\\
- [(D(v_0^{\pm}) - p_0^{\pm}I)\nu^{\pm}]_t &= \left[ L^{\beta} \partial_t u_0  + \sum_{\alpha\in \{\pm\}} B^{\alpha,\beta} v_0^{\alpha} \right]_t &\mbox{ on }& (0,T)\times \Sigma,
\end{align*}
Hence it remains to rewrite the right-hand side in terms of the quantities $K^{\pm} \in \R^{3\times 3}$ and $M^{\pm} \in \R^{3\times 3} $ (this notation was also used in \cite{gahn2025effective}) defined by 
\begin{align}\label{Def:K}
    K_{ij}^{\alpha}:= \begin{cases} 
    B_{ij}^{\alpha,\alpha}  &\mbox{ for } i,j=1,2,
    \\
    2 B_{ij}^{\alpha,\alpha} &\mbox{ for } (i=1,2, \mbox{ and } j=3) \mbox{ or } (i=3,  \mbox{ and } j=1,2),
    \\
   2 B_{ii}^{\alpha,\alpha}  &\mbox{ for } i=3,
    \end{cases}
\end{align}
and 
\begin{align}\label{Def:M}
    M_{ij}^{\alpha} :=\begin{cases} B_{ij}^{\alpha , - \alpha} &\mbox{ for } i,j=1,2,
    \\
    0 &\mbox{ for } i=3 \mbox{ or } j=3.
    \end{cases}
\end{align}
By an elemental calculation we get 
\begin{align}
\begin{aligned}\label{eq:Relation_LB_LKM}
 \int_{\Sigma} \sum_{\alpha \in \{\pm\}} & L^{\alpha}\partial_t u_0  \cdot \phi^{\alpha} + \sum_{\alpha,\beta \in \{\pm\}} B^{\alpha,\beta}v_0^{\alpha} \cdot \phi^{\beta} \, d\x
 \\
 &= \int_{\Sigma}\sum_{\pm} L^{\pm} \partial_t u_0  \cdot  \phi^{\pm} + \sum_{\pm} K^{\pm} v_0^{\pm} \cdot \phi^{\pm} + M^+ v_0^+ \cdot \phi^- + M^- v_0^- \cdot \phi^+ \, d\x.
\end{aligned}
\end{align}
The right-hand side is the representation from \cite{gahn2025effective} with the additional term including $L^{\pm} \partial_t u_0$ arising from the elastic solid. 
Plugging in this identity into $\eqref{eq:var_macro_fluid_with_B}$, we obtain the weak formulation of the macro-model $\eqref{def:Macro_Stokes_model_strong}$ for the fluid flow.

It remains to show that the jump of the normal component of the normal stress vanishes across $\Sigma$, see formula $\eqref{RHS_gamma=1}$. To  show this result, we first have to derive the macroscopic limit problem for the displacement $u_0$.

\subsubsection*{Derivation of the macroscopic equation $\eqref{eq:macro_model_displacement_gamma1}$ for the displacement $u_0$:}
We start with a representation for the corrector $u_1$ from Proposition \ref{prop:compactness_disp_gamma1}, for which we use the following cell problem: Let $\chi_{ij} \in H_{\#}^1(Z_s)^3/\R^3$ for $i,j=1,2,3$ be the unique weak solution of 
\begin{align}
\begin{aligned}\label{eq:CellProblem_chi_ij}
-\nabla_y \cdot (A (D_y(\chi_{ij}) + M_{ij})) &= 0 &\mbox{ in }& Z_s,
\\
-A(D_y(\chi_{ij}) + M_{ij} ) \nu &= 0 &\mbox{ on }& \Gamma,
\\
\chi_{ij} \mbox{ is } Y\mbox{-periodic, } & \int_{Z_f} \chi_{ij} dy = 0,
\end{aligned}
\end{align}
with $M_{ij} \in \R^{3\times 3}$ defined by 
\begin{align*}
    M_{ij}:= \frac12 (e_i \otimes e_j + e_j \otimes e_i).
\end{align*}
Due to the Korn inequality and the fact that every rigid displacement which is $Y$-periodic is constant, this problem has a unique weak solution. Now, we are able to give the representation for $u_1$:
\begin{proposition}\label{prop:representation_u1}
Let $\hat{u}_0= (u_0^1,u_0^2)$ and $u_1$ be the limit functions from Proposition \ref{prop:compactness_disp_gamma1}. Then it holds for almost every $(t,\x,y) \in (0,T)\times \Sigma \times Z_s$ that
\begin{align*}
u_1(t,\x,y) = \sum_{i,j=1}^2 D_{\x}(\hat{u}_0)_{ij} (t,\x) \chi_{ij}(y),
\end{align*}
with the cell solutions $\chi_{ij}$ defined by $\eqref{eq:CellProblem_chi_ij}$
\end{proposition}
\begin{proof}
As a test-function in the weak microscopic equation $\eqref{eq:Var_Micro_veps}$ we choose for $(t,x)  \in (0,T) \times  \oemf$ the function $\phieps(t,x):= \vareps \phi\btxfxe$ with $\phi \in C_0^{\infty}((0,T)\times \Sigma, C_{\#}^{\infty}(\overline{Z}))^3$ such that $\phi=0$ on $S^{\pm}$ and extend this function by zero to the bulk domains $\oeps^{\pm}$. From the \textit{a priori} estimates in Lemma \ref{lem:apriori_velocity_displacement} we obtain that all terms including the fluid velocity $\veps^M$ are of order $\vareps$ (the bulk terms vanish, due to the choice of $\phi$). Since $\peps^M \rats 0$ (see Proposition \ref{prop:compactness_fluid_layer}), we get for $\vareps \to 0$
\begin{align*}
\int_0^T \int_{\Sigma} \int_{Z_s} A (D_{\x}(\hat{u}_0) + D_y(u_1)) : D_y(\phi) dy d\x dt = 0.
\end{align*}
By density this equation is valid for all $\phi \in L^2((0,T)\times \Sigma, H_{\#}^1(Z_s))^3$. Since this problem has a unique solution, we obtain the desired result.
\end{proof}

Next, we derive the macroscopic equation for $u_0$. For this, we choose in $\eqref{eq:two_scale_model_fluid}$ test-functions of the form $\phi = (0,\phi^M,0)$ with (see the decomposition in $\eqref{eq:decomposition_phiM}$ with $\phi^{\pm} = 0$ and $w^M= 0$)
\begin{align*}
    \phi^M(\x,y) = \sum_{i=1}^3 \psi_i(\x) q_i^{\Gamma}(y)
\end{align*}
with $\psi \in H_0^1(\Sigma)^3$. Using the representation of $v_0^M$ from Proposition \ref{prop:representation_v_0^M}, we get  (using $D_y(\phi^M) = 0$ in $\Sigma \times Z_s$ and $D_{\x}(\phi^M) = D_{\x}(\psi)$ in $\Sigma \times Z_s$)
\begin{align}
\begin{aligned}\label{eq:aux_derivation_macro_displacement}
0 =& \int_{\Sigma} \int_{Z_f} D_y(v_0^M) : D_y(\phi^M) dy d\x + \int_{\Sigma} \int_{Z_s} A [D_{\x}(\hat{u}_0) + D_y(u_1)] : D_{\x} (\psi) dy d\x
\\
=& \sum_{i,j=1}^3 \int_{\Sigma} \partial_t u_0^i \psi_j \int_{Z_f} D_y(q_i^{\Gamma}) : D_y(q_j^{\Gamma} ) dy d\x + \sum_{\alpha \in \{\pm\}} \sum_{i,j=1}^3 \int_{\Sigma} [v_0^{\alpha}]_i \psi_j \int_{Z_f} D_y(q_i^{\alpha}): D_y(q_j^{\Gamma}) dy d\x
\\
&+ \int_{\Sigma} \int_{Z_s} A [D_{\x}(\hat{u}_0) + D_y(u_1)] : D_{\x} (\psi) dy d\x.
\end{aligned}
\end{align}
To rewrite the integral over $Z_s$ we introduce the effective elasticity tensor $A^{\ast} \in \R^{2\times 2 \times 2 \times 2}$ defined by ($i,j,l,k = 1,2$) 
\begin{align}
\begin{aligned}\label{def:effective_elasticity_tensor}
A^{\ast}_{ijkl}:&= \int_{Z_s} [A(M_{kl} + D_y(\chi_{kl}))]_{ij} dy = \int_{Z_s} [A(M_{kl} + D_y(\chi_{kl}))]:M_{ij}dy 
\\
&=\int_{Z_s} [A(M_{kl} + D_y(\chi_{kl}))]:[M_{ij} + D_y(\chi_{ij})]dy 
\end{aligned}
\end{align}
with the cell solutions $\chi_{kl}$ defined via $\eqref{eq:CellProblem_chi_ij}$. This definition is also valid for all $i,j,k,l=1,2,3$. However, if one index is $3$, then we have $A^{\ast}_{ijkl} = 0$, what follows from $-M_{3l} = D_y(-y_3 e_l)$ for $l=1,2,3$. Using the representation of $u_1$ from Proposition \ref{prop:representation_u1}, we obtain after an elemental calculation with $\hat{\psi}:= (\psi_1,\psi_2)$ (here we use that $A^{\ast}_{ijkl}=0$ if one index is $3$)
\begin{align*}
\int_{Z_s} A [D_{\x}(\hat{u}_0) + D_y(u_1)] : D_{\x} (\psi) dy  = A^{\ast}D_{\x}(\hat{u}_0):D_{\x}(\hat{\psi})
\end{align*}
almost everywhere in $\Sigma$. Now, we define the effective tensor $L^{\Gamma} \in \R^{3\times 3}$ for $i,j=1,2,3$ by
\begin{align}\label{Def:LGamma}
L^{\Gamma}_{ij}:= \int_{Z_s} D_y(q_i^{\Gamma}):D_y(q_j^{\Gamma})dy.
\end{align}
Obviously $L^{\Gamma} $ is symmetric and it is easy to check that it is also positive definite.  Altogether, we obtain from $\eqref{eq:aux_derivation_macro_displacement}$ that
\begin{align*}
    0 = \int_{\Sigma} L^{\Gamma} \partial_t u_0 \cdot \psi d\x + \sum_{\pm} \int_{\Sigma} L^{\pm} \psi \cdot v_0^{\pm}d\x + \int_{\Sigma} A^{\ast} D_{\x}(\hat{u}_0) : D_{\x}(\hat{\psi}) d\x
\end{align*}
for all $\psi \in H^1_0(\Sigma)^3$. By density this equation is also valid for all $\psi \in H_0^1(\Sigma)^2 \times L^2(\Sigma)$, since it includes no derivatives for $\psi_3$. This is the weak formulation $\eqref{eq:var_macro_model_displ_gamma1}$ of the macroscopic equation $\eqref{eq:macro_model_displacement_gamma1}.$ The initial condition $u_0(0) = 0$ follows directly from the compactness results in Proposition \ref{prop:compactness_disp_gamma1} and $\ueps(0) = 0$.  

\begin{remark}
We emphasize that the stress term in the equation above is only depending on the first and second component of the test-function $\psi$. 
\end{remark}

\begin{corollary}\label{cor:LGamma_u_0_Lpm_v_0}
Almost everywhere on $(0,T)\times \Sigma$ it holds that
\begin{align*}
    L^{\Gamma} \partial_t u_0 \cdot e_3 = - \sum_{\pm} L^{\pm}e_3 \cdot v_0^{\pm} .
\end{align*}
\end{corollary}
\begin{proof}
We test the macroscopic equation $\eqref{eq:var_macro_model_displ_gamma1}$ with  $\psi = e_3 \phi$ with $\phi \in C_0^{\infty}(\Sigma)$. 
\end{proof}

It remains to show $\eqref{RHS_gamma=1}$. For this, we use the following Lemma giving a relation between the effective coefficients $L^{\Gamma}$, $L^{\pm}$, and $B^{\alpha, \beta}$. 

\begin{lemma}\label{lem:relation_effective_coefficients_L}
It holds for $\alpha \in \{\pm\}$ 
\begin{align*}
    L^+ + L^- +  L^{\Gamma} &= 0,
    \\
     L^{\alpha} + \sum_{\beta \in \{\pm\}} \left[B^{\alpha ,\beta}\right]^T &= 0,
    \\
    (K^{\pm} + L^{\pm})e_3 &= 0.
\end{align*}
\end{lemma}
\begin{proof}
This result is an easy consequence of the relation $e_i = q_i^{\Gamma} + q_i^+ + q_i^-$ from Remark \ref{rem:linear_comb_q_i}  for $i=1,2,3$. In fact, we have for $i,j=1,2,3$ 
\begin{align*}
L^{\Gamma}_{ij} = \int_{Z_f} D_y(q_i^{\Gamma}) : D_y(q_j^{\Gamma})dy =  - \int_{Z_f} D_y(q_i^{\Gamma}) : D_y( q_j^+ + q_j^-) dy = - (L_{ji}^+ + L_{ji}^-).
\end{align*}
Using the symmetry of $L^{\Gamma}$ we get $L^+ + L^- +  L^{\Gamma} = 0$.
Similar
\begin{align*}
B_{ij}^{\alpha ,\alpha} = \int_{Z_f} D_y(q_i^{\alpha}) : D_y(q_j^{\alpha}) dy = \int_{Z_f} D_y(q_i^{\alpha}) : D_y(- q_j^{\Gamma} - q_j^{-\alpha}) dy = - L_{ij}^{\alpha} - B_{ji}^{\alpha, - \alpha}.
\end{align*}
The identity for $K^{\pm}$ and $L^{\pm}$ follows by similar arguments and we skip it.
\end{proof}
Hence, we obtain with Lemma \ref{lem:relation_effective_coefficients_L}  and Corollary \ref{cor:LGamma_u_0_Lpm_v_0} together with $\eqref{eq:Relation_LB_LKM}$ almost everywhere on $(0,T)\times \Sigma$ 
\begin{align*}
L^+ \partial_t u_0 \cdot \nu^+ - L^- \partial_t &u_0 \cdot \nu^- +  K^+v_0^+ \cdot \nu^+ -K^-v_0^- \cdot \nu^- 
\\
&= \sum_{\alpha\in \{\pm\}} L^{\alpha} \partial_t u_0 \cdot e_3 + \sum_{\alpha,\beta \in  \{\pm\}} B^{\alpha,\beta} v_0^{\alpha} \cdot e_3
\\
&= - L^{\Gamma} \partial_t u_0 \cdot e_3  + \sum_{\alpha,\beta \in  \{\pm\}} B^{\alpha,\beta} v_0^{\alpha} \cdot e_3
\\
&= \sum_{\alpha \in \{\pm\}} \left[ \left\{L^{\alpha} + \sum_{\beta \in\{\pm\}} \left[B^{\alpha,\beta}\right]^T  \right\} e_3 \right] \cdot v_0^{\alpha} = 0.
\end{align*}

Finally, we prove the statement from Remark \ref{rem:macro_model_displ_gamma1}, (ii) regarding the jump of the normal stress of the fluid across $\Sigma$, namely: It holds in the weak sense that 
\begin{align}\label{eq:interface_condition_normal_stress}
     -\nabla_{\x} \cdot (A^{\ast} D_{\x}(\hat{u}_0)) = -\llbracket (-D(v_0) + p_0 I)\nu \rrbracket \qquad\mbox{on } (0,T)\times \Sigma.
\end{align}
In fact, choosing $\psi := \theta \in H^1(\Omega,\partial_D \Omega)^3$ with $\theta|_{\Sigma} \in H^1_0(\Sigma)^3$ in $\eqref{eq:var_macro_model_displ_gamma1}$ we get with Lemma \ref{lem:relation_effective_coefficients_L}
\begin{align*}
 \int_{\Sigma} A^{\ast} D_{\x}(\hat{u}_0) : D_{\x}(\hat{\theta}) d\x &= - \int_{\Sigma} L^{\Gamma} \partial_t u_0 \cdot \theta + \sum_{\pm} L^{\pm} \theta\cdot v_0^{\pm} d\x 
 \\
 &= \int_{\Sigma} \sum_{\alpha \in \{\pm\}} L^{\alpha } \partial_t u_0 \cdot \theta + \sum_{\alpha,\beta \in\{\pm\}} [B^{\alpha,\beta}]^T \theta \cdot v_0^{\alpha} d\x
 \\
 &= \int_{\Sigma} \sum_{\alpha \in \{\pm\}} L^{\alpha } \partial_t u_0 \cdot \theta + \sum_{\alpha,\beta \in\{\pm\}} B^{\alpha,\beta} v_0^{\alpha} \cdot \theta d\x.
\end{align*}
Now, choosing $\phi^{\pm} = \theta $ in $\Omega^{\pm}$ in the variational equation $\eqref{eq:var_macro_fluid_with_B}$ for $(v_0^{\pm},p_0^{\pm})$, we get $\eqref{eq:interface_condition_normal_stress}$. Since the third component of the tensor $A^{\ast} D_{\x}(\hat{u}_0)$ is zero,  the result in \eqref{eq:interface_condition_normal_stress} is in accordance with the condition $\eqref{RHS_gamma=1}$ for the  jump of the normal component of the normal stress for the fluid across $\Sigma$.

To obtain uniqueness for the macroscopic problem we use the following (semi)-positivity result.
\begin{lemma}\label{lem:Boundary_positive_semidefinit}
For every $\xi^{\Gamma},\xi^{\pm} \in \R^n$ it holds that
\begin{align*}
  L^{\Gamma} \xi^{\Gamma} \cdot \xi^{\Gamma} + 2  \sum_{\alpha \in \{\pm\}} L^{\alpha} \xi^{\Gamma} \cdot \xi^{\alpha} + \sum_{\alpha,\beta \in \{\pm\}} B^{\alpha,\beta} \xi^{\alpha}\cdot \xi^{\beta} \geq 0.
\end{align*}
\end{lemma}
\begin{proof}
We define 
\begin{align*}
\xi^M:= \sum_{i=1}^3 \xi^{\Gamma}_i q_i^{\Gamma} + \sum_{\alpha \in\{\pm\}} \sum_{i=1}^3 \xi^{\alpha}_i q_i^{\alpha}.
\end{align*}
A similar calculation as in $\eqref{eq:aux_derivation_macro_fluid_Bpm}$ shows
\begin{align*}
0 \le \|D_y (\xi^M)\|^2_{L^2(Z_f)} =  L^{\Gamma} \xi^{\Gamma} \cdot \xi^{\Gamma} + 2  \sum_{\alpha \in \{\pm\}} L^{\alpha} \xi^{\Gamma} \cdot \xi^{\alpha} + \sum_{\alpha,\beta \in \{\pm\}} B^{\alpha,\beta} \xi^{\alpha}\cdot \xi^{\beta},
\end{align*}
which gives the desired result.
\end{proof}
\begin{remark}
In Lemma \ref{lem:Boundary_positive_semidefinit} we cannot expect to obtain positivity (compare also \cite[Lemma 5]{gahn2025effective}). In fact, choosing $\xi^{\Gamma} = \xi^+ = \xi^-$ we obtain that $\xi^M$ in the proof above is constant and therefore
\begin{align*}
    L^{\Gamma} \xi^{\Gamma} \cdot \xi^{\Gamma} + 2  \sum_{\alpha \in \{\pm\}} L^{\alpha} \xi^{\Gamma} \cdot \xi^{\alpha} + \sum_{\alpha,\beta \in \{\pm\}} B^{\alpha,\beta} \xi^{\alpha}\cdot \xi^{\beta} = 0.
\end{align*}
\end{remark}

\begin{corollary}
The solution  $(v_0,p_0,u_0)$ of the macroscopic problem  $\eqref{def:Macro_Stokes_model_strong}  \, + \, \eqref{eq:macro_model_displacement_gamma1}$ is unique.
\end{corollary}

\begin{proof}
This follows by standard arguments using Lemma \ref{lem:Boundary_positive_semidefinit}. Therefore, we just sketch the main idea. Let $(v^i_0,p^i_0,u^i_0), i=1,2$ be two solutions of the macroscopic problem $\eqref{def:Macro_Stokes_model_strong}  \, + \, \eqref{eq:var_macro_fluid_with_B}$ (where \eqref{eq:var_macro_fluid_with_B} is an equivalent formulation of \eqref{eq:macro_model_displacement_gamma1}).  We test the equations for the differences
$(\delta v_0, \delta p_0) = (v^1_0 - v^2_0,p_0^1 -p_0^2)$, respectively  $\delta u_0 = u_0^1 - u_0^2$ by $\delta v_0$ respectively $\partial_t(\delta u_0)$, add the obtained equations and use Lemma \ref{lem:Boundary_positive_semidefinit} to estimate the terms on $\Sigma$.

\end{proof}

\section{Derivation of the limit problem for $\gamma = 3$}
\label{sec:limit_problem_gamma3}

In this section, we deal with the case $\gamma = 3$. The compactness results for $\veps = (\veps^+,\veps^M,\veps^-)$ and $\peps = (\peps^+,\peps^M,\peps^-)$ remain valid. For the displacement $\ueps$ we obtain a Kirchhoff-Love displacement in the limit, see Proposition \ref{prop:compactness_disp_gamma3}, leading to a plate equation. The calculations to obtain the macroscopic fluid model for $(v_0^{\pm},p_0^{\pm})$ are very similar to the case $\gamma = 1$. The only difference is that we have to consider test-functions vanishing in $\oems$ to avoid a blow-up of the elastic stress term. Hence, following the arguments in Section \ref{sec:limit_problem_gamma1}, we obtain
\begin{align*}
\sum_{\pm} & \left\{ \int_{\Omega^{\pm}} \partial_t v_0^{\pm} \cdot \phi^{\pm} dx +\int_{\Omega^{\pm}} D(v_0^{\pm}) : D(\phi^{\pm}) dx -  \int_{\Omega^{\pm}} p_0^{\pm} \nabla \cdot \phi^{\pm} dx   \right\} 
\\
&+ \int_{\Sigma} \int_{Z_f} D_y(v_0^M) : D_y(\phi^M) dy d\x - \int_{\Sigma} \int_{Z_f} p_1^M \nabla_y \cdot \phi^M dy d\x 
 = \sum_{\pm} \int_{\Omega^{\pm}} f^{\pm}_0 \cdot \phi^{\pm} dx ,
\end{align*}
now (compared to the case $\gamma = 1$) valid for all test-functions from the space 
\begin{align*}
    \spaceH := \left\{ \phi \in \spaceV \, : \, \phi^M|_{Z_s} = 0 \right\}.
\end{align*}
We emphasize that this space was used in \cite{gahn2025effective} for the treatment of a rigid solid (with periodic boundary conditions on the lateral boundary $\partial_D \Omega$).  In particular, $(v_0^M,p_1^M)$ solve the cell problem from Proposition \ref{prop:representation_v_0^M} with $u_0 = u_0^3 e_3$, and we have 
\begin{align*}
v_0^M(t,\x,y) &=  \partial_t u_0^3 (t,\x) q_3^{\Gamma}(y) + \sum_{\alpha \in \{\pm\}} \sum_{i=1}^3  [v_0^{\alpha}]_i(t,\x,0) q_i^{\alpha}(y),
\\
p_1^M(t,\x,y) &=  \partial_t u_0^3 (t,\x) \pi_3^{\Gamma}(y) + \sum_{\alpha \in \{\pm\}} \sum_{i=1}^3  [v_0^{\alpha}]_i(t,\x,0) \pi_i^{\alpha}(y) .
\end{align*}
Further, $\eqref{eq:aux_derivation_Stokes_macro}$ and $\eqref{eq:aux_derivation_macro_fluid_Bpm}$ are still valid (with $u_0^1 = u_0^2 = 0$), and we obtain (by the same arguments as in Section \ref{sec:limit_problem_gamma1}) the interface conditions $\eqref{Macro_Model_Fluid_Stress_normal_gamma3}$ and $\eqref{Macro_Model_Fluid_Stress_tangential_gamma3}$.  We emphasize that to obtain the jump condition \eqref{Macro_Model_Fluid_Stress_normal_gamma3} the first term on the right-hand side in $\eqref{eq:Relation_LB_LKM}$ can be written by using $u_0 = u_0^3 e_3$ and Lemma \ref{lem:relation_effective_coefficients_L} (with $\phi_3 = \phi_3^+|_\Sigma = \phi_3^-|_\Sigma $)
\begin{align*}
    \sum_{\pm} L^{\pm} \partial_t u_0 \cdot \phi^{\pm} = -L_{33}^{\Gamma} \phi_3 \partial_t u_0^3 + \sum_{i=1}^2 \sum_{\pm} L^{\pm} \partial_t u_0 \cdot e_i \phi^{\pm}_i.
\end{align*}
\\

\subsubsection*{Derivation of the macroscopic equation $\eqref{eq:macro_model_displacement_gamma3}$ for the displacement $u_0^3 $ and $\hat{u}_1$:}
It remains to derive the limit equation for the displacement. Here, the idea is the same as in \cite{gahn2022derivation}, so we only give the main ideas and skip the details. 
For a representation of $u_2$ from Proposition \ref{prop:compactness_disp_gamma3}, we need, beside $\eqref{eq:CellProblem_chi_ij}$, the following cell problem: We define $\chi_{ij}^B \in H^1_{\#}(Z_s)^3$ for $i,j=1,2,3$ as the unique weak solutions of the cell problems
\begin{align}
\begin{aligned}\label{eq:cell_problem_chi_ij_B}
-\nabla_y \cdot \left( A(D_y(\chi_{ij}^B) - y_3 M_{ij})\right) &= 0 &\mbox{ in }& Z_s,
\\
-A(D_y(\chi_{ij}^B) - y_3 M_{ij})\nu &= 0 &\mbox{ on }& \Gamma,
\\
\chi_{ij}^B \mbox{ is } Y\mbox{-periodic, } & \int_{Z^f} \chi_{ij}^B dy = 0.
\end{aligned}
\end{align}
Existence and uniqueness follows again as a consequence of the Korn inequality. Now, we give the representation of $u_2$:
\begin{proposition}\label{prop:representation_u2}
Let $u_0^3$, $\hat{u}_1$, and $u_2$ be the limit functions from Proposition \ref{prop:compactness_disp_gamma3}. Then it holds for almost every $(t,\x,y) \in (0,T)\times \Sigma \times Z_s$ that
\begin{align*}
    u_2(t,\x,y) = \sum_{i,j=1}^2 \left[D_{\x}(\hat{u}_1)_{ij}(t,\x) \chi_{ij}(y) + \partial_{ij} u_0^3 (t,\x) \chi_{ij}^B(y)\right],
\end{align*}
with the cell solutions $\chi_{ij}$ and $\chi_{ij}^B$ for the cell problems $\eqref{eq:CellProblem_chi_ij}$ and $\eqref{eq:cell_problem_chi_ij_B}$, respectively.
\end{proposition}
\begin{proof}
See \cite[Proposition 6.1]{gahn2022derivation} for details and we only sketch the idea. As a test-function in the weak microscopic equation $\eqref{eq:Var_Micro_veps}$ we choose $\phieps=0$ in $\oeps^{\pm}$ and $\phieps(t,x)= \vareps^2 \phi\btxfxe$ for $(t,x)\in (0,T)\times \oeps^M$ with $\phi \in C^{\infty}_0((0,T)\times \Sigma,C_{\#}^{\infty}(\overline{Z}))^3$ with $\phi = 0$ on $S^{\pm}$. For $\vareps \to 0$ we obtain together with a density argument that almost everywhere in $(0,T)\times \Sigma$
\begin{align*}
    \int_{Z_s} A\left[D_{\x}(\hat{u}_1) - y_3 \nabla_{\x}^2 u_0^3  + D_y(u_2) \right] : D_y(\phi^M) dy = 0
\end{align*}
for all $\phi^M \in H^1_{\#}(Z_s)^3$. This implies the desired result.
\end{proof}

\begin{remark}
The proof of Proposition \ref{prop:representation_u2} shows that $u_2$ is the unique weak solution of the cell problem
\begin{align*}
- \nabla_y \cdot \left( A(D_{\x}(\hat{u}_1) - y_3 \nabla_{\x}^2 u_0^3 + D_y(u_2))\right) &= 0 &\mbox{ in }& (0,T)\times \Sigma \times Z_s,
\\
- A(D_{\x}(\hat{u}_1) - y_3 \nabla_{\x}^2 u_0^3 + D_y(u_2))\nu &= 0 &\mbox{ on }& (0,T)\times \Sigma \times \Gamma,
\\
u_2 \mbox{ is } Y\mbox{-periodic, }& \int_{Z_s} u_2 dy = 0.
\end{align*}
\end{remark}
Now, we derive the macroscopic equation for $u_0^3$ and $\hat{u}_1$. This can be done by choosing the same test-functions in $\eqref{eq:Var_Micro_veps}$ as in \cite[Section 6]{gahn2022derivation}. However, here we choose a slightly different approach more related to the choice of test-functions in Section \ref{sec:limit_problem_gamma1}. Thus, consider $\phieps^{\pm} = 0$ in $\oeps^{\pm}$ and  for $x \in \oeps^M$ let 
\begin{align*}
\phi^M_{\vareps}(x):= \psi_3(\x) q_3^{\Gamma}\left(\frac{x}{\vareps}\right) + \vareps \sum_{i=1}^2 \left( U_i(\x) - \frac{x_3}{\vareps} \partial_i \psi_3(\x) \right) q_i^{\Gamma}\left(\frac{x}{\vareps}\right),
\end{align*}
where, $\psi_3 \in H_0^2(\Sigma)$ and $\hat{U}:=(U_1,U_2) \in H_0^1(\Sigma)^2$. We have
\begin{align*}
    D(\phieps^M)(x) = \vareps \left(D_{\x}(\hat{U})(\x) - \frac{x_3}{\vareps} \nabla_{\x}^2 \psi_3(\x)\right) \qquad \mbox{ in } \oems.
\end{align*}
To compute $D(\phieps^M)$ in $\oemf$, we first use the formula $\nabla (av) = \nabla a \otimes v + a \nabla v$, for a scalar function $a$ and a vector field $v$, to obtain:
\begin{align*}
\nabla \phieps^M(x) =& \nabla_{\x} \psi_3(\x) \otimes q_3^{\Gamma}\left(\fxe\right) + \foe \psi_3(\x) \nabla_y q_3^{\Gamma}\left(\fxe\right) 
\\
&+ \vareps \sum_{i=1}^2 \left(\nabla_{\x} U_i(\x) - \frac{x_3}{\vareps} \nabla_{\x} \partial_i \psi_3(\x) - \foe \partial_i \psi_3 (\x) e_3 \right) \otimes q_i^{\Gamma}\left(\fxe\right)
\\
&+ \sum_{i=1}^2 \left(U_i (\x) - \frac{x_3}{\vareps}\partial_i \psi_3(\x)\right) \nabla_y q_i^{\Gamma}\left(\fxe\right).
\end{align*}
This implies
\begin{align*}
D(\phieps^M)(x) =& \foe \psi_3(\x) D_y(q_3^{\Gamma})\left(\fxe\right)
\\
&+ \bigg\{ sym \left( \nabla_{\x} \psi_3(\x) \otimes q_3^{\Gamma}\left(\fxe\right)\right) + \sum_{i=1}^2 \partial_i \psi_3(\x) sym\left(e_3 \otimes q_i^{\Gamma}\left(\fxe\right)\right)
\\
&\hspace{2em} + \sum_{i=1}^2 \left(U_i(\x) - \frac{x_3}{\vareps} \partial_i \psi_3(\x)\right) D_y(q_i^{\Gamma})\left(\fxe\right) \bigg\}
\\
&+ \vareps \sum_{i=1}^2 sym\left(\left(\nabla_{\x} U_i(\x) - \frac{x_3}{\vareps}\nabla_{\x} \partial_i \psi_3(\x) \right) \otimes q_i^{\Gamma}\left(\fxe\right)\right)\qquad \mbox{ in } \oemf.
\end{align*}
\\
Now, using $tr(a \otimes b) = a\cdot b$, we get with $tr D_y(q_i^{\Gamma}) = \nabla_y \cdot q_i^{\Gamma} = 0$ for $i=1,2,3$, that
\begin{align} \label{DivphiepsM}
\nabla \cdot \phieps^M (x) =& \nabla_{\x} \psi_3(\x) \cdot q_3^{\Gamma}\left(\fxe\right) + \sum_{i=1}^2 \partial_i \psi_3(\x) [q_i^{\Gamma}]_3\left(\fxe\right) \nonumber\\
&+ \vareps \sum_{i=1}^2 \left( \nabla_{\x} U_i(\x) - \frac{x_3}{\vareps}\nabla_{\x} \partial_i \psi_3(\x)\right) \cdot q_i^{\Gamma}\left(\fxe\right) \qquad \mbox{ in } \oemf.
\end{align}
Inserting now $\phieps = (0,\phieps^M,0)$ as a test function in $\eqref{eq:Var_Micro_veps}$, and using the \textit{a priori} estimates from Section \ref{sec:apriori}, we obtain 
\begin{align*}
   \foe \int_{\oemf} \vareps &D(\veps^M) : D_y(q_3^{\Gamma})\left(\fxe\right) \psi_3(\x) dx 
   \\
   &+ \foe \int_{\oems} A_{\vareps} \frac{D(\ueps)}{\vareps} : \left(D_{\x}(\hat{U}) - \frac{x_3}{\vareps} \nabla_{\x}^2 \psi_3 \right) dx = \mathcal{O}(\vareps).
\end{align*}
Hence, after multiplication with smooth time dependent functions  and integration with respect to time, for $\vareps \to 0$ we obtain from the compactness results in Section \ref{sec:compactness_results} (in particular, note that the expression \eqref{DivphiepsM} together with $\peps^M \rats 0$ implies that the pressure term vanishes for $\vareps \to 0$) {almost everywhere in $(0,T)$
\begin{align}
\begin{aligned}\label{eq:aux_deri_macro_displ_gamma3}
\int_{\Sigma} &\psi_3 \int_{Z_f} D_y(v_0^M) : D_y(q_3^{\Gamma}) dy d\x 
\\
&+ \int_{\Sigma} \int_{Z_s} A \left[ D_{\x}(\hat{u}_1) - y_3 \nabla_{\x}^2 u_0^3 + D_y(u_2)\right] : \left[ D_{\x}(\hat{U}) - y_3 \nabla_{\x}^2 \psi_3 \right] dy d\x = 0.
\end{aligned}
\end{align}
For the first term we use the representation of $v_0^M$ from Proposition \ref{prop:representation_v_0^M}, see also Remark \ref{rem:representation_v_0^M}, to obtain (with the effective coefficients $K^\pm, L^\pm,L^{\Gamma}$ defined in Section \ref{sec:limit_problem_gamma1}) almost everywhere in $(0,T)\times \Sigma$ 
\begin{align*}
 &\int_{Z_f} D_y(v_0^M) : D_y(q_3^{\Gamma}) dy = \partial_t u_0^3  \int_{Z_f} D_y(q_3^{\Gamma}) : D_y(q_3^{\Gamma} ) dy 
\\
&+ \sum_{\alpha \in \{\pm\}} \sum_{i=1}^3  [v_0^{\alpha}]_i  \int_{Z_f} D_y(q_i^{\alpha}) : D_y(q_3^{\Gamma}) dy 
= L^{\Gamma}_{33} \partial_t u_0^3 - \left[K^+ v_0^+ \cdot \nu^+  - K^- v_0^- \cdot \nu^- \right] ,
\end{align*}
where at the end we also used $L^{\pm}e_3 = -K^{\pm}e_3$, see Lemma \ref{lem:relation_effective_coefficients_L}.
For the second term on the left-hand side in $\eqref{eq:aux_deri_macro_displ_gamma3}$, we obtain with the decomposition of $u_2$ from Proposition \ref{prop:representation_u2} after an elemental calculation
\begin{align*}
 \int_{\Sigma} &\int_{Z_s} A \left[ D_{\x}(\hat{u}_1) - y_3 \nabla_{\x}^2 u_0^3 + D_y(u_2)\right] : \left[ D_{\x}(\hat{U}) - y_3 \nabla_{\x}^2 \psi_3 \right] dy d\x 
 \\
 &= \int_{\Sigma} a^{\ast} D_{\x}(\hat{u}_1) : D_{\x}(\bar{U}) + b^{\ast} \nabla_{\x}^2 u_0^3 : D_{\x}(\bar{U}) + b^{\ast} D_{\x}(\hat{u}_1) : \nabla_{\x}^2 \psi_3 + c^{\ast} \nabla_{\x}^2 u_0^3 : \nabla_{\x}^2 \psi_3 d\x
\end{align*}
with the effective coefficients $a^{\ast},b^{\ast},c^{\ast} \in \R^{2\times 2 \times 2 \times 2}$ defined by ($\alpha,\beta,\gamma,\delta = 1,2$)
\begin{align}
\begin{aligned}
\label{def:effective_coefficients_plate}
a^{\ast}_{\alpha\beta\gamma\delta} &:=  \frac{1}{\vert Z^s \vert} \int_{Z^s} A  \left(D_y(\chi_{\alpha \beta}) + M_{\alpha\beta}\right): \left(D_y (\chi_{\gamma\delta})  + M_{\gamma\delta} \right)dy,
\\
b^{\ast}_{\alpha\beta\gamma\delta} &:=   \frac{1}{\vert Z^s \vert} \int_{Z^s} A \left(D_y(\chi_{\alpha \beta}^B)  - y_3 M_{\alpha\beta} \right) : \left(D_y (\chi_{\gamma\delta})  + M_{\gamma\delta} \right)dy,
\\
c^{\ast}_{\alpha\beta\gamma\delta} &:=   \frac{1}{\vert Z^s \vert} \int_{Z^s} A  \left(D_y(\chi_{\alpha \beta}^B)  - y_3 M_{\alpha\beta} \right): \left(D_y (\chi_{\gamma\delta}^B)  -y_3 M_{\gamma\delta}\right)dy.
\end{aligned}
\end{align}
Altogether, we obtain the macroscopic plate equation in $\eqref{eq:macro_model_displacement_gamma3}$. 

\begin{corollary}
The solution  $(v_0,p_0,u_0^3,\hat{u}_1)$  of the macroscopic problem  $\eqref{def:Macro_Stokes_model_strong_gamma3}  \, + \, \eqref{eq:macro_model_displacement_gamma3}$ is unique.
\end{corollary}

\begin{proof}
For the proof of the coercivity (for suitable combinations) of the operators $a^{\ast},\, b^{\ast}$ and $c^{\ast}$ we refer to the proof of \cite[Theorem 2]{griso2020homogenization}. Hence, the uniqueness easily follows.
\end{proof}

\section{Conclusion}
\label{sec:conclusion}

To conclude, let us summarize the effective interface laws at $\Sigma$ derived in this paper and relate them to other contributions from the literature. For the fluid flow, we obtain kinematic coupling conditions consisting of the continuity of the normal component of the velocity (for $\gamma = 1,3$): 
\begin{align} \label{Cont_normal_v}
  [v_0^+]_3 = [v_0^-]_3,  
\end{align}
and slip conditions of Navier-type involving the tangential component of the effective normal stress in the fluid and the time derivative of the effective displacement $u_0$:
\begin{align}
\gamma &= 1: \quad - [(D(v_0^{\pm}) - p_0^{\pm}I)\nu^{\pm}]_t = [L^\pm\partial_t u_0]_t +[K^{\pm}v_0^{\pm}]_t + M^{\mp}v_0^{\mp}, \label{Slip_cond_gamma_1}\\
\gamma &= 3: \quad - [(D(v_0^{\pm}) - p_0^{\pm}I)\nu^{\pm}]_t = \partial_t u_0^3[L^{\pm} e_3]_t +[K^{\pm}v_0^{\pm}]_t + M^{\mp}v_0^{\mp}, \label{Slip_cond_gamma_3} 
\end{align}
Here, the effective coefficients $K^\pm, M^\pm,$ and $L^\pm$ are computed from suitable Stokes-cell problems on $Z_f$. 
Furthermore, we have transmission conditions for the normal component of the effective normal stress in the fluid:
\begin{align}
\gamma &= 1: \quad -\llbracket (D (v_0) - p_0 I ) \nu\cdot \nu\rrbracket = 0, \label{Jump_normal_comp_normal_stress_1}\\
\gamma &= 3: \quad -\llbracket (D (v_0) - p_0 I ) \nu\cdot \nu\rrbracket =  -L^{\Gamma}_{33} \partial_t u_0^3 + K^+ v_0^+ \cdot \nu^+  - K^- v_0^- \cdot \nu^-,\label{Jump_normal_comp_normal_stress_3}
\end{align}
with effective coefficients $L^\Gamma$ and $K^\pm$ being computed again from suitable Stokes-cell problems on $Z_f$. 
Whereas for the case $\gamma = 1$ the condition \eqref{Jump_normal_comp_normal_stress_1} implies continuity of the normal component of the effective normal stress across $\Sigma$, the condition \eqref{Jump_normal_comp_normal_stress_3} represents, see \eqref{Eq_displ_gamma_3_2} below, a dynamic coupling condition describing the balance of forces at $\Sigma$, namely, the effective body force on the Kirchhoff-Love plate is given by the jump in the normal component of the effective normal stress in the fluid.

The homogenization of the elasticity equations in the membrane leads to effective interface laws for the displacement $u^0$ at $\Sigma$. We obtain an evolution equation for $\gamma =1$,
\begin{align}
 L^{\Gamma} \partial_t u_0 - \nabla_{\x} \cdot (A^{\ast} D_{\x}(\hat{u}_0)) &=- \sum_{\pm} [L^{\pm}]^T v_0^{\pm},   
\end{align} 
with an effective elasticity tensor $A^{\ast}$ defined by means of suitable elasticity-cell problem on $Z_s$, 
and Kirchhoff-Love plate equations for the case $\gamma=3$:
\begin{align}
-\nabla_{\x} \cdot \left(a^{\ast} D_{\x}(\hat{u}_1) + b^{\ast} \nabla_{\x}^2 u_0^3 \right) &= 0 \label{Eq_displ_gamma_3_1},\\
\nabla_{\x}^2 : \left(b^{\ast} D_{\x}(\hat{u}_1) + c^{\ast}\nabla_{\x}^2 u_0^3 \right) &= -\llbracket (D (v_0) - p_0 I ) \nu\cdot \nu\rrbracket \label{Eq_displ_gamma_3_2},
\end{align}
with effective coefficients $a^{\ast}, b^{\ast}$, and $c^{\ast}$ computed again from elasticity-cell problems on $Z_s$. 

In our previous contribution \cite{gahn2022derivation}, we considered a fluid-structure interaction of a fluid with an elastic porous membrane with different scalings of the equations in the porous membrane. This resulted in different transmission conditions, especially for the fluid velocity at the effective interface $\Sigma$. There, all components of the fluid velocity were continuous across the interface $\Sigma$, and equal to the solid velocity $\partial_t u_0$, meaning no fluid flow relative to the solid, and thus no transversal flux through the membrane. Comparing the results of the present paper with those in \cite{gahn2025effective}, where fluid flow and transport through a rigid porous membrane was considered, assuming the same scaling in the flow equations as in the present paper, we can observe that the elastic structure contributes to the slip transmission conditions for the fluid \eqref{Slip_cond_gamma_1}-\eqref{Slip_cond_gamma_3} as well as to the jump condition in the normal component of the effective normal stress in the fluid \eqref{Jump_normal_comp_normal_stress_1}-\eqref{Jump_normal_comp_normal_stress_3}. It doesn't, however, influence the continuity of the normal component of the velocity \eqref{Cont_normal_v}. In the recent contribution \cite{kuan2024fluid} the authors study the interaction problem between the flow of an incompressible, viscous fluid modeled by the Navier-Stokes equations, and a poroviscoelastic medium modeled by the Biot equations, which are coupled across an interface with mass and elastic energy, modeled by a reticular plate equation. The coupling conditions used there (derived by phenomenological arguments) correspond in part with those derived in the present paper, like \eqref{Cont_normal_v} and \eqref{Jump_normal_comp_normal_stress_3} or are of a similar type, like \eqref{Slip_cond_gamma_1}-\eqref{Slip_cond_gamma_3}. We remark, however, that the homogenization procedure gives us the possibility to derive effective coefficients with an explicit dependence on the microscopic structure of the thin membrane. The coupling between a free fluid and a poroelastic medium across a thin membrane (modelling the endothelial layer in the blood vessels) was also considered in \cite{SilvaJaeger2020}, where again coupling conditions across the membrane are formulated by heuristic arguments. Here again the continuity of the normal component of the velocity \eqref{Cont_normal_v} is assumed at the interface. However, for the tangential component of the velocity a no-slip condition is considered. Let us mention, that the rigorous derivation of interface conditions for the interaction of a free fluid and a porous/poroelastic medium separated by a porous/poroelastic membrane are still missing. The techniques developed in this paper for the derivation of interface conditions between two fluid bulk domains separated by an elastic porous membrane, based on (two-scale) convergence in the bulk domains and in the porous membrane, are a first step towards more complex problems, like those mentioned above.

\renewcommand{\theequation}{\thesection.\arabic{equation}}
\setcounter{equation}{0} %counter zurücksetzen

\renewcommand{\thelemma}{\thesection.\arabic{lemma}}
\setcounter{lemma}{0} %counter zurücksetzen

\renewcommand{\thedefinition}{\thesection.\arabic{definition}}
\setcounter{definition}{0} %counter zurücksetzen

\renewcommand{\theremark}{\thesection.\arabic{remark}}
\setcounter{remark}{0} %counter zurücksetzen

\renewcommand{\thecorollary}{\thesection.\arabic{corollary}}
\setcounter{corollary}{0} %counter zurücksetzen

\begin{appendix}
\section{Auxiliary inequalities}

We summarize some important inequalities for perforated layers from the literature.  We start with the following well known Korn and Poincar\'e inequality for functions vanishing on $\geps$.

\begin{lemma}[Korn and Poincar\'e inequality]\label{lem:Korn_inequality}
There exists a constant $C>0$, such that for all $\ueps \in H^1(\oemf)^n$ with $\ueps = 0$ on $\geps$ it holds that
\begin{align*}
    \|\ueps\|_{L^2(\oemf)} \le C\vareps \|\nabla \ueps \|_{L^2(\oemf)} \le C\vareps \|D(\ueps)\|_{L^2(\oemf)}.
\end{align*}
\end{lemma}

Next, we give the trace inequality on $\geps$, which is also well-known and can be obtained by a simple domain decomposition.
\begin{lemma}[Standard scaled trace inequality]\label{lem:Stand_scaled_trace_inequality}
For all $\theta >0$ there exists $C_\theta>0$  independent of $\vareps$, such that for all $\phieps \in H^1(\Omega_\vareps^{M,\ast}), \ast \in \{s,f\}$, it holds that
\begin{align*}
    \|\phieps\|_{L^2(\geps)} \le \frac{C_\theta}{\sqrt{\vareps}} \|\phieps\|_{L^2(\oeps^{M,\ast})}
    + \theta \sqrt{\vareps}\|\nabla \phieps \|_{L^2(\oeps^{M,\ast})}.
\end{align*}

\end{lemma}

For functions in the thin layer with zero boundary conditions on the lateral boundary we have the following Korn inequality on thin perforated domains:
\begin{lemma}[Korn inequality thin perforated layers]\label{KornInequalityPerforatedLayer}
For all $\weps \in H^1(\Omega_{\epsilon}^{M,\ast})^3$ for $\ast \in \{s,f\}$ with $\weps = 0 $ on $ \partial_D \Omega_{\epsilon}^{M,\ast}$ it holds that 
\begin{align*}
\sum_{i=1}^2 \foe \Vert \weps^i\Vert_{L^2(\Omega_{\epsilon}^{M,\ast})} + \sum_{i,j=2}^2 \foe\Vert \partial_i \weps^j \Vert_{L^2(\Omega_{\epsilon}^{M,\ast})} + \Vert \weps^3\Vert_{L^2(\Omega_{\epsilon}^{M,\ast})} +  \Vert \nabla \weps \Vert_{L^2(\Omega_{\epsilon}^{M,\ast})} \le \frac{C}{\epsilon} \Vert D(\weps)\Vert_{L^2(\Omega_{\epsilon}^{M,\ast})}.
\end{align*}
\end{lemma}
\begin{proof}
    See \cite[Theorem 2]{GahnJaegerTwoScaleTools} or \cite[Theorem 6]{griso2020homogenization}.
\end{proof}
Now, we formulate an embedding result on $H^1$, which allows to control the $L^2$-norm in the thin layer by its gradient and the $L^2$-norm in the bulks:
\begin{lemma}\label{lem:estimate_L2_membrane_Bulk_Gradient}
For every $\phieps \in H^1(\oemf)$ it holds that
\begin{align}\label{ineq:aux_trace_inequality}
        \frac{1}{\sqrt{\vareps}} \|\phieps\|_{L^2(\oemf)} \le C \left( \sqrt{\vareps} \|\nabla \phieps\|_{L^2(\oemf)} + \|\phieps\|_{L^2(S_{\vareps}^{\pm})}\right)
    \end{align}
This inequality is also valid if we replace $\oemf$ with $\oems$ and $S_{\vareps}^{\pm}$ with $\geps$.
In particular, for every $\phieps = (\phieps^+,\phieps^M,\phieps^-) \in H^1(\oef)$ we obtain
\begin{align*}
\frac{1}{\sqrt{\vareps}} \|\phieps\|_{L^2(\oemf)} \le C \left( \sqrt{\vareps} \|\nabla \phieps^M\|_{L^2(\oemf)} + \sum_{\pm} \|\phieps\|_{H^1(\oeps^{\pm})}\right).
\end{align*}
\end{lemma}
\begin{proof}
See \cite[Lemma A.3]{gahn2025effective}.
\end{proof}

The following Lemma provides an extension operator for $H^1$-functions in the thin perforated layers with an explicit control for the gradient and the symmetric gradient. This results can be found for example in \cite{GahnJaegerTwoScaleTools}.
\begin{lemma}[Extension operator]\label{lem:Extension_operators}
 There exists an extension operator $E_{\vareps}^s: H^1(\oems) \rightarrow H^1(\oeps^M)$ such that for all $\phieps \in H^1(\oems)$ it holds that
\begin{align*}
    \|E_{\vareps}^s \phieps\|_{L^2(\oeps^M)} &\le C \left(\|\phieps\|_{L^2(\oems)} + \vareps \|\nabla \phieps\|_{L^2(\oems)}\right)  
    \\
    \|\nabla E_{\vareps}^s \phieps \|_{L^2(\oeps^M)} &\le C\|\nabla \phieps\|_{L^2(\oems)}, 
    \\
    \vareps\|D(E_{\vareps}^s \phieps)\|_{L^2(\oeps^M)} &\le C \|D(\phieps)\|_{L^2(\oems)},
\end{align*}
for a constant $C>0$ independent of $\vareps$. 
\end{lemma}

Finally, we give a Bogovoskii result. Here, the crucial point is the result in the thin layer.
\begin{lemma}\label{lem:Bogovskii}\
\begin{enumerate}[label = (\roman*)]
\item For every $h_{\vareps}^{\pm} \in L^2(\oeps^{\pm})$ there exists $\weps^{\pm} \in H^1(\oeps^{\pm},\partial_D \oeps^{\pm} \cup S_{\vareps}^{\pm})^3$ such that
\begin{align*}
\nabla \cdot \weps^{\pm} = h_{\vareps}^{\pm}, \qquad \|\weps^{\pm}\|_{H^1(\oeps^{\pm})} \le C \|h_{\vareps}^{\pm}\|_{L^2(\oeps^{\pm})}.
\end{align*}

\item For every $h_{\vareps}^M \in L^2(\oemf)$ there exists $\weps^M \in H^1(\oemf,\partial_D \oemf \cup \geps)^3$ such that
\begin{align*}
\nabla \cdot \weps^M = h_{\vareps}^M , \qquad \foe \|\weps^M\|_{L^2(\oemf)} + \|\nabla \weps^M \|_{L^2(\oemf)} \le C \|h_{\vareps}^M\|_{L^2(\oemf)}.
\end{align*}
In particular, there exists $\weps \in H^1(\oef,\partial_D \oef \cup \geps)^3$ with $\nabla \cdot \weps = h_{\vareps}^M$ in $\oemf$ and
\begin{align*}
    \foe \|\weps\|_{L^2(\oef)} + \|\nabla \weps\|_{L^2(\oef)} \le C \|h_{\vareps}^M\|_{L^2(\oemf)}.
\end{align*}
\end{enumerate}
\end{lemma}
\begin{proof}
The first result is classical. The second one can be found in Step 4 in the proof of \cite[Lemma 4.2]{gahn2022derivation}.
\end{proof}

\renewcommand{\theequation}{\thesection.\arabic{equation}}
\setcounter{equation}{0} %counter zurücksetzen

\renewcommand{\thelemma}{\thesection.\arabic{lemma}}
\setcounter{lemma}{0} %counter zurücksetzen

\renewcommand{\thedefinition}{\thesection.\arabic{definition}}
\setcounter{definition}{0} %counter zurücksetzen

\renewcommand{\theremark}{\thesection.\arabic{remark}}
\setcounter{remark}{0} %counter zurücksetzen

\renewcommand{\thecorollary}{\thesection.\arabic{corollary}}
\setcounter{corollary}{0} %counter zurücksetzen

\section{Two-scale convergence}
\label{sec:two_scale_convergence}

In this section we briefly repeat the definition of the two-scale convergence for thin (perforated) layers, see \cite{NeussJaeger_EffectiveTransmission}, and summarize some compactness results, in particular related to thin elastic materials, see \cite{GahnJaegerTwoScaleTools} and \cite{griso2020homogenization}. 
\begin{definition}[Two-scale convergence in $\oem$]
We say the sequence $\weps \in L^2((0,T)\times \oem)$ converges (weakly) in the two-scale sense to a limit function $w_0 \in L^2((0,T)\times  \Sigma \times Z)$ if 
\begin{align*}
\lim_{\vareps\to 0} \foe \int_0^T \int_{\oem} \weps(t,x) \phi \left(t,\x,\fxe\right) dxdt = \int_0^T\int_{\Sigma} \int_Z w_0(t,\x,y) \psi(t,\x,y) dy d\x dt
\end{align*}
for all $\phi \in L^2((0,T)\times \Sigma,C_{\#}^0(\overline{Z}))$. We write $\weps \rats w_0$.
\end{definition}
First of all we notice, that for every sequence $\weps \in L^2((0,T), H^1(\oem))$ fulfilling
\begin{align*}
    \frac{1}{\sqrt{\vareps}} \|\weps \|_{L^2((0,T)\times \oem)} \le C,
\end{align*}
there exists a two-scale convergent subsequence. If we have additional bounds for the gradients, we have the following well-known compactness result, see for example \cite{GahnEffectiveTransmissionContinuous}.
\begin{lemma}\label{LemmaBasicTSCompactness}\
For every sequence $\weps \in L^2((0,T), H^1(\oem))$ with
\begin{align*}
\frac{1}{\sqrt{\vareps}}\| \weps \|_{L^2((0,T)\times \oem)} + \sqrt{\vareps}\| \nabla \weps \|_{L^2((0,T)\times \oem)}  \le C,
\end{align*}
there exists $w_0 \in  L^2( (0,T)\times \Sigma, H_{\#}^1(Z)/\R)$, such that up to a subsequence
\begin{align*}
\weps  \rats w_0,\qquad 
\vareps \nabla \weps \rats \nabla_y w_0.
\end{align*}
\end{lemma}
For the displacement of the elastic solid we usually have a bound for the symmetric gradient. In this case we have the following two-scale compactness result:
\begin{lemma}\label{lem:two_scale_Kirchhoff_Love}
Let $\weps\in L^2((0,T),H^1(\oems,\partial_D \oems))^3$  be a sequence with
\begin{align*}
\|D(\weps)\|_{L^2((0,T)\times \oems)} \le C \vareps^{\frac32}.
\end{align*}
Then  there exist $w_0^3 \in L^2((0,T),H^2_0(\Sigma))$, $\hat{w}_1=(w_1,w_2) \in L^2((0,T),H^1_0(\Sigma))^2$,  and $w_2 \in L^2((0,T)\times \Sigma, H_{\#}^1(Z_s)/\R)^3$, such that up to a subsequence (for $\alpha = 1,2$)
\begin{align*}
\chi_{\oems}\weps^3 &\rats \chi_{Z_s}w_0^3,
\\
\chi_{\oems}\frac{\weps^{\alpha}}{\epsilon} &\rats \chi_{Z_s}\big({w}_1^{\alpha} - y_3 \partial_{\alpha} w_0^3\big),
\\
\frac{1}{\epsilon} \chi_{\oems} D(\weps) &\rats  \chi_{Z_s} \left(D_{\x}(\hat{w}_1) - y_3 \nabla_{\x}^2 w_0^3 + D_y(w_2) \right).
\end{align*}
\end{lemma}
\begin{proof}
See \cite[Theorem 3]{GahnJaegerTwoScaleTools}. 
\end{proof}

\end{appendix}

\bibliographystyle{abbrv}
\bibliography{literature}

\end{document}